\pdfoutput=1
\documentclass{amsart}

\usepackage{mathtools}
\usepackage{amssymb}
\usepackage{mathrsfs}

\newtheorem{theorem}{Theorem}[subsection]
\newtheorem{proposition}[theorem]{Proposition}
\newtheorem{corollary}[theorem]{Corollary}
\newtheorem{lemma}[theorem]{Lemma}

\theoremstyle{definition}
\newtheorem{example}[theorem]{Example}

\theoremstyle{remark}
\newtheorem{remark}[theorem]{Remark}

\usepackage{microtype}
\usepackage{enumitem}
    \setlist{noitemsep}
\usepackage{tikz-cd}
    \usetikzlibrary{decorations.markings}
    \makeatletter
    \tikzcdset{
        open/.code     = {\tikzcdset{hook, circled};},
        closed/.code   = {\tikzcdset{hook, slashed};},
        open'/.code    = {\tikzcdset{hook', circled};},
        closed'/.code  = {\tikzcdset{hook', slashed};},
        circled/.code  = {\tikzcdset{markwith = {\draw (0,0) circle (.375ex);}};},
        slashed/.code  = {\tikzcdset{markwith = {\draw[-] (-.4ex,-.4ex) -- (.4ex,.4ex);}};},
        crossed/.code  = {\tikzcdset{markwith = {\draw[-] (-1ex, 1ex) -- (1ex,-1ex);\draw[-] (-1ex,-1ex) -- (1ex,1ex);}}},
        markwith/.code = {
            \pgfutil@ifundefined%
            {tikz@library@decorations.markings@loaded}%
            {\pgfutil@packageerror{tikz-cd}{You need to say %
                \string\usetikzlibrary{decorations.markings} to use arrows with markings}{}}{}%
            \pgfkeysalso{/tikz/postaction = {
                /tikz/decorate,
                /tikz/decoration={markings, mark = at position 0.5 with {#1}}}
            }
        },
    }
    \makeatother
    \tikzcdset{
        arrow style=tikz,
        diagrams={>={Straight Barb[scale=0.8]}}
    }
\usepackage{hyperref}

\usepackage[backend    = biber,
            citestyle  = alphabetic,
            bibstyle   = alphabetic,
            sortlocale = auto,
            doi        = false,
            isbn       = false,
            url        = false,
            datamodel  = ext-eprint,
            backref    = true]{biblatex}
\input{biblatex-eprint.cfg}
\AtEveryBibitem{\clearlist{language}}

\DeclareBibliographyCategory{sigles}
\addtocategory{sigles}{EGAIII1,SGA3I,SGA4III,EGAII,%
    EGAIV4,EGAIV2,EGAInew,stacks-project,SGA7II,SGA5}
\renewcommand*{\bibfont}{\small}

\newcommand{\rmT}{\mathrm{T}}
\newcommand{\rmB}{\mathrm{B}}
\newcommand{\rmm}{\mathrm{m}}
\newcommand{\rmd}{\mathrm{d}}
\newcommand{\rmH}{\mathrm{H}}
\newcommand{\bfA}{\mathbf{A}}
\newcommand{\bfC}{\mathbf{C}}
\newcommand{\bfS}{\mathbf{S}}
\newcommand{\bfB}{\mathbf{B}}
\newcommand{\bfP}{\mathbf{P}}
\newcommand{\bfV}{\mathbf{V}}
\newcommand{\bfN}{\mathbf{N}}
\newcommand{\bfZ}{\mathbf{Z}}
\newcommand{\scrD}{\mathscr{D}}
\newcommand{\scrO}{\mathscr{O}}
\newcommand{\scrU}{\mathscr{U}}
\newcommand{\scrT}{\mathscr{T}}
\newcommand{\calS}{\mathcal{S}}
\newcommand{\calH}{\mathcal{H}}
\newcommand{\calD}{\mathcal{D}}
\newcommand{\calE}{\mathcal{E}}
\newcommand{\bbG}{\mathbb{G}}
\newcommand{\frakm}{\mathfrak{m}}
\newcommand{\fraka}{\mathfrak{a}}
\newcommand{\frakU}{\mathfrak{U}}
\newcommand{\frakV}{\mathfrak{V}}
\newcommand{\frakS}{\mathfrak{S}}

\newcommand{\into}{\hookrightarrow}
\newcommand{\opp}{\text{opp}}
\newcommand{\inv}{^{-1}}
\newcommand{\frno}{n\textsuperscript{$\circ$}}
\newcommand*\leftidx[2]{\vphantom{#2}#1\kern-\scriptspace#2}
\makeatletter
\DeclareRobustCommand{\simto}{\mathrel{\mathpalette\@verto\sim}}
\DeclareRobustCommand{\congto}{\mathrel{\mathpalette\@verto\cong}}
\newcommand{\@verto}[2]{%
  \lower.5\p@\vbox{
    \lineskiplimit\maxdimen
    \lineskip-.5\p@
    \ialign{%
      $\m@th#1\hfil##\hfil$\crcr
      #2\crcr
      \to \crcr
    }%
  }%
}
\makeatother
\newcommand{\etale}{{\text{\'et}}} %
\newcommand{\fppf}{{\text{fppf}}}  %

\DeclareMathOperator{\Res}{Res} %
\DeclareMathOperator{\Image}{Im} %
\DeclareMathOperator{\Pic}{Pic} %
\DeclareMathOperator{\ad}{ad} %
\DeclareMathOperator{\End}{End} %
\DeclareMathOperator{\Fr}{Fr} %
\DeclareMathOperator{\Spec}{Spec} %
\DeclareMathOperator{\Sym}{Sym} %
\DeclareMathOperator{\Hom}{Hom} %
\DeclareMathOperator{\cHom}{\calH\mathit{om}}
\DeclareMathOperator{\cSpec}{\calS\mathit{pec}}
\DeclareMathOperator{\cEnd}{\calE\mathit{nd}}
\DeclareMathOperator{\id}{id} %
\DeclareMathOperator{\rk}{rk} %
\DeclareMathOperator{\Der}{Der} %
\DeclareMathOperator{\cDer}{\calD\mathit{er}}
\DeclareMathOperator{\Nm}{Nm}

\newcommand{\citeSP}[1]{%
\cite[\href{http://stacks.math.columbia.edu/tag/#1}{Tag~#1}]{stacks-project}}
\newcommand{\catHiggs}{\mathbf{Higgs}}
\newcommand{\catLoc}{\mathbf{LocSys}}
\newcommand{\catPic}{\mathbf{Pic}}
\newcommand{\Higgs}{\text{\normalfont Higgs}}
\newcommand{\dR}{\text{\normalfont dR}}
\newcommand{\catMod}{\mathbf{Mod}}

\newcommand{\catQCoh}{\mathbf{QCoh}}

\newcommand{\catVect}{\mathbf{Vect}}

\newcommand{\catSch}{\mathbf{Sch}}
\newcommand{\catSet}{\mathbf{Set}}
\newcommand{\cotsp}[1]{\rmT^\ast(#1)}
\newcommand{\SPtag}[1]{\href{https://stacks.math.columbia.edu/tag/#1}{Tag~#1}}

\makeatletter
    \newcommand{\pushright}[1]{%
    \ifmeasuring@#1\else\omit\hfill$\displaystyle#1$\fi\ignorespaces}
    \newcommand{\pushleft}[1]{%
    \ifmeasuring@#1\else\omit$\displaystyle#1$\hfill\fi\ignorespaces}
\makeatother

\numberwithin{equation}{section}

\title{%
    A Simpson Correspondence for Abelian Varieties in characteristic $p>0$%
}
\author{Yun Hao}
\address{Freie Universit\"at Berlin, Arnimallee 3, 14195, Berlin, Germany}
\email{haoyun.math@gmail.com}

\begin{document}

\maketitle

\begin{abstract}
    Let $X/k$ be an abelian variety over an algebraically closed field $k$ of
    characteristic $p > 0$. In this paper, using the Azumaya property of the
    sheaf of crystalline differential operators and the Morita equivalence, we
    show that \'etale locally over the Hitchin base, the moduli stack of Higgs
    bundles on the Frobenius twist $X'$ is equivalent to that of local systems
    on $X$. We follow the approach of \cite{zbMATH06666980}.
\end{abstract}

\section{Introduction}
Let $X/\bfC$ be a smooth projective variety over the complex numbers. In
\cite{zbMATH00165866}, Simpson established an equivalence between the category
of local systems (vector bundles with integrable algebraic connections,
equivalently (\emph{Riemann-Hilbert correspondence}), finite dimensional
representations of the fundamental group $\pi_1(X^{\text{an}})$) and that of the
semi-stable Higgs bundles whose Chern classes are zero. The correspondence
between Higgs bundles and local systems can be viewed as a Hodge theorem for
nonabelian cohomology. The theory is hence called the \emph{non-abelian Hodge
theory}, and sometimes is called the \emph{Simpson correspondence}.

There are many attempts to search for such a correspondence in positive
characteristic. From now on, let $X/S$ be a smooth scheme over a scheme $S$ of
characteristic $p > 0$. In the work of \cite{zbMATH05277424},
\citeauthor{zbMATH05277424} established such a correspondence for nilpotent
objects. More precisely, under the assumption that a lifting  of $X'/S$ modulo
$p^2$ exists, where $X'$ is the Frobenius twist of $X$, they construct a Cartier
transform from the category of modules with flat connections nilpotent of
exponent%
\footnote{
    There is a tiny difference of conventions between the definition in
    \cite[Definition~5.6]{zbMATH03350023} and the definition used in
    \cite{zbMATH05277424}. In the latter, these modules are said to be of
    nilpotent $\leq p -1$, as they are supported on the $(p - 1)$-st
    infinitesimal neighbourhood of the zero section of the cotangent bundle,
    see \S\ref{ssec:BNR-Higgs} for how this makes sense.
}
$\leq p$ to the category of Higgs modules nilpotent of exponent $\leq p$.
There is an alternative approach to this result in \cite{zbMATH06466318}.
A generalisation of this work to higher level arithmetic differential operators
(in the sense of \cite{zbMATH00913811}) is given in \cite{zbMATH05690817}.
Some other recent related works include \cite{zbMATH06544774},
\cite{zbMATH06750889} and \cite{1705.06241}.

In his work \cite{zbMATH06666980}, \citeauthor{zbMATH06666980} gave a full
version of this correspondence for (orbi)curves $X$ over an algebraically closed
field $S = \Spec k$.
At the same time, \citeauthor{zbMATH06526258} \cite{zbMATH06526258} further
generalised this correspondence for curves but between the category of
$G$-Higgs bundels and $G$-local systems, where $G$ is a reductive group over an
algebraically closed field $k$.

In the present work, following the approach of \cite{zbMATH06666980}, using the
Azumaya property the sheaf of crystalline differential operators proved by
\citeauthor{zbMATH05578707} in \cite{zbMATH05578707}, we generalise the result
from curves to abelian varieties. To be precise, we proved the following result.

\begin{theorem}[Theorem~\ref{thm:main}]
    Let $X/k$ be an abelian variety over an algebraically closed field $k$ of
    characteristic $p > 0$. Denote by $\catHiggs_{X'/k, r}$ the stack of rank $r$
    Higgs bundles on the Frobenius twist $X'$ of $X$, and by $\catLoc_{X/k, r}$
    the stack of rank $r$ local systems on $X$. Then we have the following
    results.
    \begin{enumerate}
        \item Each of the two stacks admits a natural map to a common base
            scheme
            \[
                B'_r:= \Spec \Sym^\bullet\left(
                \bigoplus_{i = 1}^r\Gamma(X, \Sym^i\Omega_{X/k}^1)\right)^\vee.
            \]
        \item There is a closed subscheme $\widetilde{Z}$ of the cotangent
            bundle of $X'_{B'_r}/B'_r$, that is finite flat over $X_{B'_r}'$,
            and a $\bfP$-equivariant isomorphism
            \[
                C\inv_{X/k}: \bfS \times^{\bfP} \catHiggs_{X'/k, r}
                \longrightarrow \catLoc_{X/k, r}
            \]
            over $B'_r$, where $\bfP = \catPic_{\widetilde{Z}/B'_r}/B'_r$ is the
            relative Picard stack and $\bfS/B'_r$ is a $\bfP$-torsor.
        \item In particular, there is an
            \'etale surjective map $U \to B'_r$, such that
            \[
                \catHiggs_{X'/k, r} \times_{B'_r} U \simeq \catLoc_{X/k, r}
                \times_{B'_r} U.
            \]
    \end{enumerate}
\end{theorem}

The idea of the proof is as follows. We know from \cite{zbMATH05578707} that the
sheaf of crystalline differential operators defines an Azumaya algebra
$\scrD_{X/k}$ over the cotangent bundle $\cotsp{X'/k}$ of $X'$
(Theorem~\ref{thm:azumaya-property}). We can also view Higgs bundles on $X'$ as
quasi-coherent $\scrO$-modules on the cotangent bundle $\cotsp{X'/k}$, supported
on a closed subscheme called the \emph{spectral cover}
(\S\ref{ssec:spectral-cover} and Proposition~\ref{prop:Higgs-spectral-cover}).
Meanwhile, local systems on $X$ can be identified with certain
$\scrD_{X/k}$-modules on the cotangent bundle $\cotsp{X'/k}$, via their
$p$-curvatures (Proposition~\ref{prop:local-system-spectral-cover}). If there is
a splitting of the Azumaya algebra $\scrD_{X/k}$ on the spectral cover, then the
Morita theory will give an equivalence between these $\scrO$-modules and
$\scrD_{X/S}$-modules on $\cotsp{X'/k}$
(Proposition~\ref{thm:splitting-principle}).

The existence of splittings in the curve case is guaranteed by Tsen's theorem
(see \cite[\S3.4]{zbMATH06666980}). For abelian varieties, the existence of
splittings of $\scrD_{X/S}$ over (the formal neighbourhood of) a spectral cover,
is obtained by observing that the stack of splittings of $\scrD_{X/S}$ is
equivalent to the stack $\catPic_{X/k}^\natural$ of line bundles with a flat
connection on $X$ (Proposition~\ref{cor:pic-equiv-to-splitting}), and that
$\catPic_{X/k}^\natural/B'_1$ is \emph{smooth}.

One of the main differences between the curve case and the higher dimensional
case is that in the latter case, the Hitchin map is not surjective and the
spectral cover is \emph{not flat} over $X'$ in general
(Proposition~\ref{prop:proper-of-universal-spectral-curve} and
Remark~\ref{rmk:can-be-empty-non-flat}). However, since the cotangent bundle of
an abelian variety is trivial, for any given spectral cover, we can construct a
larger cover, that is finite, flat, locally of finite presentation over $X'$
(Example~\ref{exa:computation-av}). This larger cover will play an important
role in the construction and the proof.

Actually the proof is rather formal. We indeed prove some results in a more
general setting, using the language of $(k, R)$-Lie algebras (reviewed in
\S\ref{sec:local}) and Lie algebroids (developed in \S\ref{sec:D}). Many of
results have already been studied in \cite{zbMATH06358801}, but the
corresponding spectral data were not considered there. By working in such a
general setting, we can see that some of the known results (e.g.,
Proposition~\ref{prop:reduced-char-polyn}) also hold for flat
$\lambda$-connections (cf. \cite[Prop.~3.2]{zbMATH01588211}), not only for flat
$1$-connections, which is in fact a consequence of an easy observation
Lemma~\ref{lem:char-descent}. It allows us to have a more abstract way to look
at these results. Moreover, the Azumaya property of the crystalline differential
operator, proven in \cite{zbMATH05578707} and \cite{zbMATH03417620}, is reproved
in detail in this article.

\subsection*{Acknowledgement}
The author is indebted to Michael Groechenig for providing him the problem
worked out in this article, and for multiple enlightening discussions. I also
benefited from discussions with H\'el\`ene Esnault, Lei Zhang and Hao Zhang are
also very helpful. The author also thanks the Berlin Mathematical School for the
financial support during his study.

\section{Generalities}

\paragraph{Conventions:}
\begin{itemize}[wide]
    \item Unadorned tensor products, $\Hom$ or $\cHom$'s, symmetric powers and
        exterior powers are taken over the structure sheaf $\scrO_X$ of a base
        scheme or over a base ring $R$, which will be clear from context.
    \item Given a scheme $X$ and a quasi-coherent $\scrO_X$-algebra $A$, $\cSpec
        A$ stands for the relative (or global) spectrum. For any quasi-coherent
        $\scrO_X$-module $E$ of over a scheme $X$, set $\bfV(E):=\cSpec \Sym
        ^\bullet E^\vee$. This convention differs from that in
        \cite[\S1.7]{EGAII} by a dual.
    \item For two natural numbers $r$ and $d$, set $S(r,d):=\binom{d+r-1}{r}$.
\end{itemize}

\subsection{Frobenius}
\label{ssec:frobenius}
Fix a scheme $S$ of characteristic $p>0$ and let $f: X \to S$ be an $S$-scheme.
Denote by $\Fr_S$ (resp.\ $\Fr_X$) the \emph{absolute Frobenius} on $S$ (resp.\
$X$), by $X':= X^{(p)}:= X \times_{S, \Fr_S} S$ the \emph{Frobenius twist} of
$X$ over $S$, and by $F_{X/S}:=\Fr_{X/S} = (\Fr_X, f): X \to X'$ the
\emph{relative Frobenius}. In case $X= \Spec R$ and $S = \Spec k$ are both
affine, where $k$ is a commutative ring of characteristic $p > 0$ and $R$ a
$k$-algebra, we can introduce corresponding maps on their coordinate rings. In
other words, we have the following commutative diagrams:%
\footnote{
    Here we use the same notations $f$, $f'$ and $w$ to denote different maps.
    However, the meaning of these maps can be inferred from context.
}
\begin{equation}
    \begin{tikzcd}[column sep = large]
        X \ar[r, "F_{X/S}"] \ar[dr, "f"']
        & X' \ar[r, "w"] \ar[d, "f'"] \ar[dr, "\square" description, phantom]
        & X \ar[d, "f"] \ar[from = ll, bend left = 45, "\Fr_X", crossing over]
        \\
        & S \ar[r, "\Fr_S"] & S,
    \end{tikzcd}\qquad
    \begin{tikzcd}[column sep = large]
        R & R' \ar[l, "\varphi_{R/k}"']
        \ar[dr, "\square" description, phantom]
        & R\ar[l, "w"'] 
        \ar[ll, bend right= 45, "\varphi_R"', crossing over]
        \\
        & k \ar[ul, "f"] \ar[u, "f'"]
        & k \ar[l, "\varphi_k"'] \ar[u, "f"].
    \end{tikzcd}
    \label{eq:notation-Frob}
\end{equation}

Recall that we have the following facts about the relative Frobenius morphism.

\begin{proposition} \label{lem:norm-of-Frobenius}
    Suppose that $f:X \to S$ is smooth of relative dimension $d > 0$. Then the
    followings hold.
    \begin{enumerate}[wide]
        \item The relative Frobenius $F_{X/S}$ is a universal homeomorphism and
            is \emph{finite locally free} (in the sense of
            \cite[\SPtag{02K9}]{stacks-project}) of rank $p^d$. In particular,
            it is faithfully flat.
        \item Denote by $ \Nm_{F_{X/S}}: (F_{X/S})_\ast \scrO_{X} \to
            \scrO_{X'}$ the \emph{norm map} (see e.g.,
            \cite[\SPtag{0BD2}]{stacks-project}). Then the composition
            \[\begin{tikzcd}
                    (F_{X/S})_\ast \scrO_{X} \ar[r, "\Nm_{F_{X/S}}"]
                    & \scrO_{X'} \ar[r, "F_{X/S}^\natural"]
                    & (F_{X/S})_\ast \scrO_{X}
            \end{tikzcd}\]
            is the $p^d$-th power map, i.e., for any open subset $U \subseteq
            X'$ and any section $g \in (F_{X/S, \ast}\scrO_X)(U) =
            \scrO_X(F_{X/S}\inv U)$, we have $F_{X/S}^\natural\Nm_{F_{X/S}}(g) =
            g^{p^d} \in \scrO_X(F_{X/S}\inv U)$.
    \end{enumerate}
\end{proposition}

\begin{proof}
    Since $X/S$ is smooth of relative dimension $d$, Zariski locally we have the
    following commutative diagram (\cite[\SPtag{054L}]{stacks-project}):
    \begin{equation} \label{eq:etale-coordinate}
        \begin{tikzcd}
            X \ar[d, "f"'] \ar[from = r, open']
            & U \ar[r, "\text{\'etale}", "g"'] \ar[d, "f|_U"']
            & \bfA_V^d. \ar[dl] \\
            S \ar[from = r, open'] & V.
        \end{tikzcd}
    \end{equation}
    where $U$ and $V$ are (affine) opens in $X$ and $S$ respectively. We then
    have $F_{U/V} =  g^\ast (F_{\bfA_V^d/V}) \circ F_{U/\bfA_V^d}$. Recall that
    $g$ being \'etale implies that $F_{U/\bfA_{V}^d}$ is an isomorphism
    (\cite[\SPtag{0EBS}]{stacks-project}), so we reduce the problem to check
    both statements for  $\bfA_V^d/V$. So without loss of generality, write $S =
    \Spec A$ and $X = \Spec A[t_1, \ldots, t_d]$. Then the two statements
    follows from easy verification (possibly by an induction on the dimension
    $d$ and by applying the norm formula for successive finite finite free
    extensions \cite[Chaptitre~III, \S9, \frno4, Propostion~6]{zbMATH05069335}).
\end{proof}

\subsection{Azumaya algebras and the Morita equivalence}
\label{ssec:azumaya-morita}

Fix a base scheme $(X, \scrO_X)$. Recall that
An \emph{Azumaya algebra} (or a \emph{central separable algebra})
$A$ over $X$ is a sheaf of $\scrO_X$-algebras (not necessarily commutative),
such that there is an \'etale surjective map
(or equivalently, an fppf map; see \cite[Prop.~IV.2.1]{zbMATH03674235})
$g: U \to X$, such that
\[
    g^\ast A \simeq \cEnd_{\scrO_U}(\scrO_U^{\mathop{\oplus}r})
\]
for some $r \in \bfN$. This is a generalisation of \emph{central simple
algebras} over a field $k$. Clearly, an Azumaya algebra is locally free of rank
$r^2$ (see e.g., \cite[\SPtag{05B2}]{stacks-project}). Given an Azumaya algebra
$A$ over $X$, a \emph{splitting} of $A$ consists of a morphism $g: T \to X$,
a locally free $\scrO_T$-module $P$, and an isomorphism $\alpha: g^\ast A \simto
\cEnd_{\scrO_T}(P)$. Splittings of an Azumaya algebra form a $\bbG_{\rmm}$-gerbe
over $(\catSch/X)$ (see~\cite[\S\S12.3.5 and 12.3.6]{zbMATH06561920} or
\cite[Chapitre V, \S4.2]{zbMATH03358638}).

Given a splitting a splitting $(T, P, \alpha)$ of an Azumaya algebra $A$, recall
that the Morita theory asserts that the functor $\catMod_{\scrO_X} \to
\catMod_{A}$, $E \mapsto P \otimes E$ defines an equivalence between the
category of $\scrO_T$-modules and that of left $A$-modules. See for example
\cite[(8.12)]{zbMATH05585185}.

\subsection{Characteristic polynomial of a twisted endomorphism}
\label{ssec:characteristic-polynomial}
Let $(X, \scrO_X)$ be a fixed scheme. Similar discussions can be found in
\cite[\S6.4]{EGAII}.

Let $E, K$ be two \emph{finite locally free} $\scrO_X$-modules and suppose that
$\rk(E) = r$. Consider a homomorphism $\varphi: E \to E \otimes K$ of
$\scrO_X$-modules, called a ($K$-)\-\emph{twisted endomorphism} of $E$. We will
also view $\varphi$ as an element in $\Gamma(X, \cEnd(E) \otimes K)$.
We have the canonical inclusion to degree $1$ map $K \to \Sym^\bullet K$. So any
twisted endomorphism $\varphi$ gives rise to an $\scrO_X$-linear map $E \to E
\otimes \Sym^\bullet K$. Then using the adjoint pair
\[
    \Hom_{\scrO_X}(E, E \otimes \Sym^\bullet K) \simeq
    \Hom_{\Sym^\bullet K}(E \otimes \Sym^\bullet K, E \otimes \Sym^\bullet K),
\]
we can view $\varphi$ as an $(\Sym^\bullet K)$-endomorphism of the locally free
module $E \otimes \Sym^\bullet K$. So it has an obviously well-defined
\emph{characteristic polynomial} $\chi_\varphi \in \Gamma(X, \Sym^\bullet K[t])$
in the usual sense. We usually write $\chi_\varphi(t) = t^r + \cdots + (-1)^i
a_it^i + \cdots + (-1)^ra_n$, where $a_i \in \Gamma(X, \Sym^i K)$.

Finally, substituting $t$ by $\varphi$ in $\chi_\varphi$,
we get an $\scrO_X$-linear map
\[
    \chi_\varphi(\varphi) = \sum_{i = 0}^r (-1)^i a_i \varphi^{r-i}:
    E \to E \otimes \Sym^r K,
\]
where $\varphi^i$ is the $\scrO_X$-linear map $E \to  E \otimes \Sym^i K$
obtained by composition
\[\begin{tikzcd}[column sep = small]
        E \ar[r, "\varphi"] & E \otimes K \ar[r, "\varphi \otimes \id"] &
        E \otimes K^{\otimes 2} \ar[r] & \cdots \ar[r] & E \otimes K^{\otimes i}
        \ar[r, twoheadrightarrow] & E \otimes \Sym^i K.
\end{tikzcd}\]
Now as in the classical case, the Cayley-Hamilton theorem reads as follows.

\begin{proposition} \label{prop:C-H}
    The map
    \begin{equation}
        \chi_\varphi(\varphi) : E \longrightarrow E \otimes \Sym^r K
        \label{eq:C-H}
    \end{equation}
    is the \emph{zero} map.
\end{proposition}

\subsection{Affine morphisms and finite morphisms}
\label{ssec:affine-morphism}
Recall the following standard results that we will use repeatedly (see e.g.,
\cite[\S\S1.3--1.4, 1.7, 6.1]{EGAII}).

\begin{proposition}\label{prop:affine-morphisms}
    \begin{enumerate}[wide]
        \item \label{prop:affine-morphisms-1}
            Let $f: Y \to X$ be an affine morphism, then the direct image
            functor $f_\ast: \catQCoh(\scrO_Y) \simeq \catQCoh(\scrO_X, A)$, $N
            \mapsto f_\ast N$, where $\catQCoh(\scrO_X, A)$ is the category
            $A$-modules which are $\scrO_X$-quasicoherent. defines an
            equivalence of categories, whose quasi-inverse is denote by $M
            \mapsto \widetilde{M}$.

        \item \label{prop:affine-morphisms-2}
            Let $X$ be a scheme and $E$ a locally free sheaf of finite rank on
            $X$. For any morphism $f: T \to X$, we have
            \begin{equation}
                \Hom_X(T, \bfV(E))
                = \Hom_{\catMod_{\scrO_X}}(E^\vee, f_\ast \scrO_T) =
                \Gamma(T, f^\ast E).
                \label{eq:section-of-vector-bundles}
            \end{equation}
            That's the reason why we define $\bfV(E)$ as the spectrum of the
            symmetric algebra of its \emph{dual}: the sections of $\bfV(E)$ are
            sections of $E$, with our notation.
        \item \label{prop:affine-morphisms-3}
            \emph{(\cite[Prop.~6.1.12]{EGAII},
            \cite[Prop.~12.13]{zbMATH05585185})} Let $f: Y \to X$ be a
            finite morphism. Then the equivalence in
            \ref{prop:affine-morphisms-1} induces an equivalence of categories%
            \begin{equation}
                f_\ast: \catVect_m(Y) \longrightarrow
                \catVect_m(X, f_\ast \scrO_Y)
                \label{eq:equivalence-finite-locally-free}
            \end{equation}
            where $\catVect_m(X)$ stands for the category of locally free
            sheaves of rank $m$ on $X$ and $\catVect_m(X, f_\ast \scrO_Y)$
            stands for that of locally free $f_\ast \scrO_Y$-modules on $X$.
    \end{enumerate}
\end{proposition}
The first two are commonly known but we just record them here for later use. But
one need to take care that last result is \emph{not} tautological, but rests
essentially on the assumption that $f$ is finite. The latter category consists
of objects that are $f_\ast \scrO_Y$-modules $E$ on $X$ such that there exists a
Zariski cover $U_i \to X$, such that $E|_{U_i} \simeq
(\scrO_Y|_{f\inv(U_i)})^{\oplus m}$ for all $i$; in other words, we can
trivialise a vector bundle on $Y$ using a finer enough covering of $X$.

\section{Local theory: restricted $(k, R)$-Lie algebras}
\label{sec:local}

Let $k$ be a commutative ring (not necessarily a field) and $R$ a commutative
$k$-algebra. We will also call a Lie algebra over $k$ a $k$-Lie algebra, to
distinguish it from a $(k, R)$-Lie algebra introduced below. Moreover, any
associative $k$-algebra is understood as a $k$-Lie algebra with the commutator
as the bracket, unless otherwise specified.

\subsection{$(k, R)$-Lie algebras}
Recall \cite{zbMATH03184072} that a \emph{$(k, R)$-Lie algebra}%
\footnote{
    Other names include \emph{Lie-Rinehart algebra}, \emph{Lie algebroid},
    \emph{differential Lie-$R$-algebra}, \emph{pseudo-alg\'ebre de Lie} or
    \emph{$d$-Lie ring}.
}
is a triple $(L, [-, -], \rho)$, where $L$ is a left $R$-module, $(L, [-, -])$
is a Lie algebra over $k$, and $\rho: L \to \Der_k(R)$ is a homomorphism of
$k$-Lie algebras and a homomorphism of left $R$-modules such that
\[
    [x, ry]  = r[x, y] + \rho(x)(r)y,
\]
for all $x, y \in L$ and $r \in R$. Sometimes we just write $x(r)$ instead of
$\rho(x)(r)$ for simplicity.

\begin{example}\label{exa:k-R-Lie-alg}
    \begin{enumerate}[wide]
        \item \label{exa:k-R-Lie-alg-1}
            A typical example is $\Der_k(R)$ itself with the bracket defined
            by $[\partial_1, \partial_2] := \partial_1 \circ \partial_2 -
            \partial_2 \circ \partial_1 \in \Der_{k}(R)$ for any two
            $k$-derivation $\partial_1$ and $\partial_2$ of $R$, and the anchor
            given by $\rho:=\id$.

        \item \label{exa:k-R-Lie-alg-2}
            Given a $(k, R)$-Lie algebra $(L, [-, -], \rho)$ and an element
            $\lambda \in k$, we set
            \[
                \leftidx{^\lambda}[x, y]:= \lambda [x, y],
                \text{\quad and \quad}
                \leftidx{^\lambda}\rho(x) := \lambda \rho(x),
            \]
            for any $x, y \in L$. We can easily verify that $(L,
            \leftidx{^\lambda}[-, -], \leftidx{^\lambda}\rho)$ is again a $(k,
            R)$-Lie algebra, which we will simply write as $\leftidx{^\lambda}L$
            for short. Tautologically, $\leftidx{^1}L= L$ (in particular,
            $\leftidx{^1}\Der_k(R) = \Der_k(R)$). Note that the $R$-module
            structures on $L$ stay the same for all $\leftidx{^\lambda}L$ with
            different $\lambda \in k$.

        \item \label{exa:k-R-Lie-alg-4}
            Let $L$ be an arbitrary left $R$-module. The \emph{trivial} $(k,
            R)$-Lie algebra structure on $L$ is the zero bracket and the zero
            anchor. One easily verifies all axioms are satisfied.
    \end{enumerate}
\end{example}

\subsection{Modules over a $(k, R)$-Lie algebra}
Given a $(k, R)$-Lie algebra $L$, an \emph{$L$-module}%
\footnote{
    Other names include \emph{$(k,R)$-Lie algebra representation of $L$},
    \emph{(flat) $L$-connection}, \emph{$L$-co-connection} or
    \emph{$L^\vee$-connection}, where $L^\vee := \Hom_R(L, R)$ is the $R$-module
    dual of $L$.
}
$M$ is an $R$-module as well as an $L$-module (where $L$ is viewed as a $k$-Lie
algebra) in a compatible way. More precisely, writing $\sigma_R: R \to
\End_k(M)$, $r \mapsto (m \mapsto rm:=\sigma_R(m))$ for the $R$-module structure
on $M$ and
\[
    \sigma_L : L \to \End_k(M)
\]
the $L$-module structure, which is a homomorphism of $k$-Lie algebras, we then
require that $\sigma$ to be an $R$-module homomorphism, and that
\begin{equation}
    \sigma_L(x)(rm) = \rho(x)(r)m + r \sigma_L(x)(m)\in M,
    \label{eq:compatible}
\end{equation}
for all $x \in L$, $r \in R$ and $m \in M$. For simplicity, we
sometimes just write $x(m)$ instead of $\sigma_L(x)(m)$.

Analogously to the case of $k$-Lie algebras, there is a \emph{universal
enveloping algebra} $U_{k, R}(L)$ associated to a $(k, R)$-Lie algebra $L$,
which is a unital associative $k$-algebra, together with maps
\[
    \iota_R: R \to U_{k, R}(L),
    \text{\quad and \quad}
    \iota_L: L \to U_{k, R}(L),
\]
where $\iota_R$ is a morphism of $k$-algebras, inducing a left $R$-module
structure on $U_{k, R}(L)$, and $\iota_L$ is a morphism of $k$-Lie algebras as
well as a morphism left $R$-modules. This algebra has the following universal
property.

\begin{proposition}
    \label{prop:universal-property-of-U(L)-0}
    Let $\sigma_R: R \to A$ be a morphism of associative $k$-algebras (inducing
    a left $R$-module structure on $A$), and $\sigma_L: L \to A$ a morphism
    $k$-Lie algebras as well as a morphism of left $R$-modules, such that
    $[\sigma_L(x), \sigma_R(r)] = \sigma_R( \rho(x)(r))$, for all $x \in L$ and
    $r \in R$, then there is a unique homomorphism $\mu: U_{k, R}(L) \to A$,
    that is a morphism of unital associative $k$-algebras, a morphism of left
    $R$-modules as well as a morphism of $k$-Lie algebras, making the following
    diagrams commutative:
    \[
        \begin{tikzcd}
            R \ar[dr, "\sigma_R"'] \ar[r, "\iota_R"]
            & U_{k, R}(L) \ar[d, dotted, "\mu"] \\
            & A,
        \end{tikzcd}\qquad
        \begin{tikzcd}
            L \ar[dr, "\sigma_L"'] \ar[r, "\iota_R"]
            & U_{k, R}(L) \ar[d, dotted, "\mu"]\\
            & A.
        \end{tikzcd}
    \]
\end{proposition}

In particular, applying the universal property to $A:= \End_k(M)$ for
$L$-modules $M$, we obtain the following proposition.

\begin{corollary}
    \label{cor:universal-property-of-U(L)}
    Let $L$ be a $(k, R)$-Lie algebra. Then we have an equivalence between the
    category of $L$-modules and that of left $U_{k, R}(L)$-modules.
\end{corollary}

The construction of $U_{k, L}(R)$ can be found in \cite[\S2]{zbMATH03184072},
which we omit here.%
\footnote{
    However, it is worth to point out that, in many references, $U_{k, L}(R)$ is
    \emph{incorrectly defined} or refered as a quotient of the universal
    enveloping algebra of the $k$-Lie algebra $L \oplus R$. The correct
    definition should rather be a \emph{subquotient} as in
    \cite{zbMATH03184072}.
}
The universal property follows from the construction and see
\cite[\S1.6]{zbMATH04147014} for a proof.

\begin{example}\label{exa:k-R-Lie-Mod-and-U}
    \begin{enumerate}[wide]
        \item \label{exa:k-R-Lie-Mod-and-U-0}
            The maps $\sigma_R: R \to \End_k(R)$ given by homothies and
            $\sigma_L:= \rho : L \to \Der_k(R) \subseteq \End_k(M)$, make $R$ a
            $U_{k, R}(L)$-module.

        \item \label{exa:k-R-Lie-Mod-and-U-1}
            If $L= \leftidx{^0}\Der_k(R)$, then $U_{k, R}(L) = \Sym_R L$. If
            $L = \leftidx{^1}\Der_k(R) = \Der_k(R)$, then $U_{k, R}(L)$ is the
            ring $D_{R/k}$ of \emph{crystalline differential operators} in case
            $R/k$ is smooth. We will discuss this later in more details in
            Example~\ref{exa:typical-example}.

        \item \label{exa:k-R-Lie-Mod-and-U-2}
            Suppose that $R/k$ is a smooth algebra of relative dimension $d$.
            Then $\Omega_{R/k}^1$ is a Zariski locally free $R$-module of rank
            $d$ and $\Der_{k}(R) \simeq \Hom_{R}(\Omega_{R/k}^1, R)$. Let $M$ be
            an $R$-module and $\lambda \in k$. Then a
            $\leftidx{^\lambda}\Der_k(R)$-module structure on $M$, or
            equivalently a $U_{k, R}(\leftidx{^\lambda}(R))$-module structure on
            $M$, is equivalent to a flat $\lambda$-connection on $M$. See
            Example~\ref{exa:typical-example} for more details.

        \item \label{exa:k-R-Lie-Mod-and-U-3}
            Let $L$ be a left $R$-module, equipped with the \emph{trivial} $(k,
            R)$-Lie algebra structure (see
            Example~\ref{exa:k-R-Lie-alg}.\ref{exa:k-R-Lie-alg-4}). Then $U_{k,
            R}(L) = \Sym_R L$ is the symmetric algebra. An $L$-module (a.k.a.
            $U_{k, R}(L)$-module) is sometimes called a $L$-co-Higgs module.
            When $L = \Der_k(R)$, this reduces to the case mentioned in
            (\ref{exa:k-R-Lie-Mod-and-U-2}) above.

        \item \label{exa:k-R-Lie-Mod-and-U-4}
            If $R = k [t_1, \ldots, t_d]$ is the polynomial algebra in $d$
            variables over $k$, then $U_{k, R}
            \left(\leftidx{^\lambda}\Der_k(R)\right)$ is the $\lambda$-twisted
            $d$-th Weyl algebra%
            \footnote{
                This is not a standard terminology. Usually only the case
                $\lambda = 1$ is considered.
            }
            $\leftidx{^\lambda}W_d(k)$, i.e., the associative $k$-algebra
            generated by $2d$ variables $t_1, \ldots, t_d, \partial_1, \ldots,
            \partial_d$, modulo the two-sided ideal generated by $t_it_j - t_j
            t_i$, $\partial_i\partial_j - \partial_j \partial_i$ and $\partial_i
            t_j - t_j \partial_i - \lambda\delta_{ij}$, $1\leq i, j \leq d$ and
            $\delta_{ij}$ is the Kronecker delta.
    \end{enumerate}
\end{example}

\subsection{Restricted $(k, R)$-Lie algebras}
Suppose that $k$ is moreover of characteristic $p > 0$.
\begin{lemma}\label{lem:identities}
    Let $A$ be a unital associative $k$-algebra, which can be naturally
    viewed as an Lie algebra over $k$ with bracket defined by bracket. Then the
    following identity holds:
    \begin{align}
        (x+y)^p & = x^p + y^p + \sum_{r= 1}^{p-1} s_r(x, y),
        \label{eq:Jacobson} \\
        \intertext{%
            where $s_r(x, y)$ (called the \emph{universal Lie polynomials}) is
            the coefficient of $t^{r-1}$ in the formal expression $\ad(tX +
            Y)^{p - 1}(X)$, divided by $r$. Suppose moreover that $x$ and
            $\ad_y^n(x)$, $n \in \bfN$ mutually commute, then
        }
        (xy)^p & = x^p y^p + \ad_{xy}^{p-1}(x)y, \label{eq:Hochschild}\\
        & = x^p y^p - x \ad_y^{p-1}(x^{p-1})y. \label{eq:Deligne}
    \end{align}
\end{lemma}
\begin{proof}
    These three identities are due to 
    Artin-Zassenhaus (\cite[\S6]{zbMATH02509896}) and
    Jacobson\cite[\S2]{zbMATH02522765},
    Hochschild (\cite[Lemma~1]{zbMATH03110239}), and
    Deligne (\cite[Proposition~5.3]{zbMATH03350023}%
    \footnote{
        Note that there is a sign error in the formula in
        \cite[Proposition~5.3]{zbMATH03350023}.
    }) respectively. See cited literatures for proofs.
\end{proof}

Recall \cite{zbMATH02502285} that a \emph{restricted $k$-Lie algebra} is a
$k$-Lie algebra $(L, [-, -])$ together with a map $-^{[p]}: L \to L$, called the
\emph{$p$-operator}, satisfying
\begin{enumerate}
    \item Jacobson's identity \eqref{eq:Jacobson}, and
    \item $\ad_{x^{[p]}} = \ad_x^p\in \End_k(L)$, for all $x \in L$, and
    \item $(ax)^{[p]} = a^p x^{[p]} \in L$, for all $a \in k$, $x \in L$.
\end{enumerate}
Similarly, a \emph{restricted $(k, R)$-Lie algebra} is a quadruple $(L, [-,-],
-^{[p]}, \rho)$, where $(L, [-,-], -^{[p]})$ is a restricted $k$-Lie algebra and
$(L, [-,-], \rho)$ is a $(k, R)$-Lie algebra, such that the anchor $\rho: L \to
\Der_k(R)$ is also a homomorphism of restricted $k$-Lie algebras, and that the
$R$-module structure and the $p$-operator are compatible in the sense that (cf.\
Hochchild's identity~\eqref{eq:Hochschild})
\[
    (rx)^{[p]} = r^p x^{[p]} + (\rho(rx))^{p-1}(r)x \in L,
\]
for any $r \in R$ and $x \in L$. We may sometimes suppress $\rho$ from the
notation. The universal enveloping algebra of a restricted $(k, R)$-Lie algebra
is the universal enveloping algebra of the underlying $(k, R)$-Lie algebra
(cf.~Example~\ref{exa:restricted-k-R-Lie}.\ref{exa:restricted-k-R-Lie-3} below).

\begin{example}\label{exa:restricted-k-R-Lie}
    \begin{enumerate}[wide]
        \item \label{exa:restricted-k-R-Lie-1}
            For any derivation $\partial \in \Der_k(R)$, its $p$-th iteration,
            as a $k$-linear endomorphism of $R$, is again a derivation, which we
            denote it by $\partial^{[p]}$. Then $(\Der_k(R), [-,-], -^{[p]},
            \id)$ is a typical example of restricted $(k, R)$-Lie algebra.
        \item \label{exa:restricted-k-R-Lie-2}
            For any given restricted $(k, R)$-Lie algebra $(L, [-,-], -^{[p]},
            \rho)$ and $\lambda \in k$, set
            \[
                x^{\leftidx{^\lambda}[p]} := \lambda^{p-1} x^{[p]}, \quad
                \forall\, x \in L.
            \]
            Then one verifies that $\leftidx{^\lambda}L := (L,
            \leftidx{^\lambda}[-,-], -^{\leftidx{^\lambda}[p]},
            \leftidx{^\lambda}\rho)$ is again a restricted $(k, R)$-Lie algebra.
        \item \label{exa:restricted-k-R-Lie-3}
            The universal enveloping algebra of a restricted $(k, R)$-Lie
            algebra has a natural restricted $(k, R)$-Lie algebra structure,
            where the $p$-operator is given by the $p$-th power map. However,
            the map $\iota_L: L \to U_{k, R}(L)$ is not a morphism of restricted
            $(k, R)$-Lie algebras in general. But as in the case of restricted
            $k$-Lie algebras, one can form the quotient $U_{k, R}^{[p]}(L):=
            U_{k, R}(L)/\langle (\iota_L(x))^p - \iota_L(x^{[p]}) \rangle$,
            called the \emph{restricted universal enveloping algebra} of $L$.
            Then the composition $L \to U_{k, R}(L) \to U_{k, R}^{[p]}(L)$
            becomes a homomorphism of restricted $(k, R)$-Lie algebras. However,
            in this article, we will not use it.
    \end{enumerate}
\end{example}

\begin{proposition}\label{prop:p-linear}
    Let $L$ be a \emph{restricted} $(k, R)$-Lie algebra. Let $A$, $\sigma_R: R
    \to A$ and $\sigma_L: L \to A$ be as in
    Proposition~\ref{prop:universal-property-of-U(L)-0}.
    Then the $k$-module homomorphism
    \[
        \psi: L \to A, \quad
        x \mapsto (\sigma_L(x))^p - \sigma_L(x^{[p]}),
    \]
    which measures how far $\sigma_L$ is from being a homomorphism of restricted
    $k$-Lie algebras, satisfies the following properties:
    \begin{enumerate}
        \item The map $\psi$ is a $\varphi_R$-linear map of $R$-modules, i.e.,
            $\psi(rx + y) = r^p \psi(x) + \psi(y)$, for all $x, y \in L$ and $r
            \in R$.
        \item $[\sigma_L(L), \sigma_R(R)] = 0$.
        \item $[\psi(L), \sigma_L(L)] = 0$, hence in particular, $[\psi(L),
            \psi(L)] = 0$.
        \item The image $\psi(L)$ lies in the centraliser $Z_{A}(\sigma_R(R),
            \sigma_L(L))$, and that this centraliser is a \emph{commutative}
            subring of $A$. Hence if in addition that $A$ is generated as a
            $k$-algebra by $\sigma_R(R)$ and $\sigma_L(L)$ (e.g., $A = U_{k,
            R}(L)$), then $\psi(L)$ lies in the centre of $A$.
    \end{enumerate}
\end{proposition}

\begin{proof}
    The proof is just a routine check using the axioms on restricted $(k,
    R)$-Lie algebras and the identities in Lemma~\ref{lem:identities}.
\end{proof}

With Proposition~\ref{prop:p-linear}, the following two corollaries are
immediate.

\begin{corollary}\label{cor:induce-maps-from-Sym}
    The map $\psi$ in Proposition~\ref{prop:p-linear} induces an $R$-linear
    map, resp.\ $R'$-linear map,
    \begin{align}
        \psi'': L \otimes_{R, \varphi_R} R  \to A, & \quad
        x \otimes r \mapsto \sigma_R(r)\psi(x),
        \label{eq:psi-L-A''}\\
        \text{resp. \ \quad}
        \psi': L \otimes_{R, w} R' \to A, & \quad
        x \otimes r' \mapsto \sigma_R(\varphi_{R/k}(r')) \psi(x),
        \label{eq:psi-L-A'}
    \end{align}
    where $\varphi_R$ and $\varphi_{R/k}$ are as in introduced in
    diagram~\eqref{eq:notation-Frob}. Moreover, the image of $\psi'$ lies in the
    centraliser $Z_A\left(\sigma_R(R), \sigma_L(L)\right)$ hence
    induces a morphism of commutative rings
    \begin{equation*}
        \psi': \Sym_{R'} (L \otimes_{R, w} R') \longrightarrow
        Z_A\left(\sigma_R(R), \sigma_L(L)\right).
    \end{equation*}
    The image of $\psi''$ lies in the centraliser of $R$ in $A$ hence
    induces a morphism of commutative rings
    \begin{equation*}
        \psi'':
        \Sym_{R'}(L \otimes_{R, w} R') \otimes_{R', \varphi_{R/k}} R \simeq
        \Sym_R (L \otimes_{R, \varphi_R} R) \longrightarrow Z_A(\sigma_R(R)).
    \end{equation*}
\end{corollary}

\begin{corollary}\label{cor:p-linear-universal}
    Let $L$ be a restricted $(k, R)$-Lie algebra. Then the map
    \[
        \psi:L \longrightarrow U_{k, R}(L),\quad x \longmapsto (\iota_L(x))^p -
        \iota_L(x^{[p]})
    \]
    satisfies the properties in Proposition~\ref{prop:p-linear}. Moreover,
    $\psi$ induces morphisms
    \begin{align}
        \psi': \Sym_{R'} (L \otimes_{R, w} R') & \longrightarrow
        Z_{U_{k, R}(L)}\left({U_{k, R}(L)}\right) \label{eq:induced-1}, \\
        \text{and\quad}
        \psi'': \Sym_R (L \otimes_{R, \varphi_R} R) & \longrightarrow Z_{U_{k,
        R}(L)}(\iota_R(R)),\label{eq:induced-2}
    \end{align}
    where $\psi'$ is a $R'$-linear and $\psi''$ is $R$-linear.
\end{corollary}

Next we state an easy lemma for later use.
\begin{lemma} \label{lem:char-descent}
    Let $L$ be a restricted $(k, R)$-Lie algebra and $M$ an $L$-module. Suppose
    moreover that $L$ and $M$ are both finite free as $R$-modules. By applying
    \eqref{eq:psi-L-A''} to $A:=\End_R(M)$, we obtain an $R$-module morphism $L
    \otimes_{R, \varphi_R} R \longrightarrow \End_R(M)$, which is equivalent
    (since all $R$-modules are assumed to be free) to an $R$-module morphism
    \[
        M \longrightarrow M \otimes_R (L^\vee \otimes_{R, \varphi_R} R).
    \]
    Then the characteristic polynomial of this map, which a priori has
    coefficients in $\Sym^\bullet (L^\vee \otimes_{R, \varphi_R} R)$, has the
    form $\chi' \otimes_{R, \varphi_R} 1 = \chi \otimes_{R', \varphi_{R/k}} 1$,
    where $\chi \in \Sym^\bullet L^\vee$ and $\chi \in \Sym^\bullet_{R'} (L^\vee
    \otimes_{R, w} R')$.
\end{lemma}

\begin{proof}
    This actually follows from definition: choose base for the finite free
    $R$-modules $M$ and $L$ respectively and the dual basis for $L^\vee:=
    \Hom_R(L, R)$. By writing down the matrix representation for the map $M \to
    M \otimes_R (L^\vee \otimes_{R, \varphi_{R}} R)$, can we easily obtain the
    result.
\end{proof}

\subsection{The Weyl algebra in characteristic $p>0$}
Let $k$ be a commutative ring of characteristic $p>0$.
In this section, we study the universal enveloping algebra for the $(k, R)$-Lie
algebra $\Der_k(R)$ with $R = k[t_1, \ldots, t_d]$, that is, the $d$-th Weyl
algebra
\[
    W_d(k) := \frac{k\langle t_1, \ldots, t_d, \partial_1, \ldots, \partial_d
    \rangle}%
    {([t_i,t_j],\,[\partial_i, \partial_j],\, [\partial_i, t_j] -\delta_{ij})}
\]
as mentioned in
Example~\ref{exa:k-R-Lie-Mod-and-U}.\ref{exa:k-R-Lie-Mod-and-U-4}.
In this situation, we have $R' = k [t_1',\ldots, t_d']$ and that the relative
Frobenius is given by $\varphi_{R/k}: R' \to R$, $t_i' \mapsto t^p$.

Recall \cite[\S5]{zbMATH05948094} that for a left $A$-module $M$, where $A$ is a
ring (not necessarily commutative), the centraliser (or commutant) $Z_A(M)$ and
the double centraliser (or bicommutant) $Z_A^2(M)$ are defined as $Z_A(M):=
Z_{\End_\bfZ(M)}(A_M)$ and $Z_A^2(M):= Z_A(Z_A(M))$, where $A_M$ stands for the
subring of $\End_\bfZ(M)$ of $A$-homothies on $M$. An $A$-module $M$ is said to
be \emph{balanced} if $A_M = Z^2_A(M)$.

The next result show that $W_d(k)$ is an Azumaya algebra over its centre when
$k$ is of characteristic $p > 0$. This result appeared in
\cite[Theorem~2]{zbMATH03417620}, \cite[Lemma~1.2]{zbMATH00814550} (only for $d
= 1$, but the proof is the same for $d > 1$) and
\cite[Thm.~2.2.3]{zbMATH05578707}. Here we want to remark that we do not need to
assume that $k$ is an algebraically closed field, nor even a field for the
following theorem to be true.

\begin{theorem}\label{thm:azumaya-local}
    Write $D:= W_d(k)$. Then $D$ is naturally a \emph{left} $D^e:=(D \otimes_Z
    D^\opp)$-module: $(\delta \otimes \delta') \cdot x := \delta x\delta'$ for
    all $\delta$, $\delta'$ and $x \in D$. All the other module structures on
    $D$ mentioned below are induced by this one by restriction.
    \begin{enumerate}[wide]
        \item \label{thm:azumaya-local-1}
            The ring $D$ has centre $Z:=Z_{D}(D)=k[t_1^p,\ldots, t_d^p,
            \partial_1^p, \ldots, \partial_d^p]$ and $D$ is a free (left and
            right) $Z$-module of rank $p^{2d}$. Moreover, $Z$ is natually an
            $R'$-algebra via $R' \to Z$, $t'_r \mapsto t^p_r$, for any $1 \leq r
            \leq d$.

        \item \label{thm:azumaya-local-2}
            The centraliser and double centraliser of $R$ in $D$ are both $A:=
            Z_{D}(R) = k [t_1, \ldots, t_d, \partial_1^p, \ldots,
            \partial_d^p]$. Moreover, $D$ is a free (left and right) $A$-module
            of rank $p^d$.%

        \item \label{thm:azumaya-local-3}
            The two maps in \eqref{eq:induced-1} and \eqref{eq:induced-2} are
            isomorphisms. In particular, we have that $A \simeq Z \otimes_{R',
            \varphi_{R/k}} R$, in other words, a push-out diagram
            \[\begin{tikzcd}
                A & Z \ar[l] \\
                R \ar[u] & R' \ar[u] \ar[l, "\varphi_{R/k}"']
            \end{tikzcd}\]
            and that $A/Z$ is faithfully flat and finitely presented since
            $\varphi_{R/k}$ is.

        \item \label{thm:azumaya-local-4}
            For each maximal ideal $\frakm \subseteq Z$ with $\kappa:=
            \kappa(\frakm) = Z/\frakm$, $D \otimes_Z \kappa$ is a central simple
            algebra over $\frakm$. In other words, the ring $D$ is an Azumaya
            algebra (central separable algebra) over $Z$.

        \item \label{thm:azumaya-local-5}
            The centraliser and double centraliser of $D$ as a $(D \otimes_Z
            A)$-module are given by
            \begin{align*}
                Z_{D \otimes_Z A}(D) \simeq A^\opp \simeq A,
                \quad\text{and}\quad
                Z^2_{D \otimes_Z A}(D) = \End_A(D).
            \end{align*}
            So the induced left $Z_{D \otimes A}(D)$-module structure on $D$ is
            the same as the right $A$-module structure on $D$. Hence $D$ is free
            as a left $Z_{D \otimes A}(D)$-module of rank $p^d$, in particular,
            it is finitely generated.

        \item \label{thm:azumaya-local-6}
            As a ($D \otimes_Z A)$-module $D$ is a balanced. That is to say,
            the ring of $(D \otimes_Z A)$-homothies of $D$ coincides with the
            double centraliser $Z^2_{D \otimes_Z A}(D) = \End_A(D)$. This fact
            also implies that $D$ is an Azumaya algera. We have a commutative
            diagram of rings
            \[\begin{tikzcd}
                    D \otimes_Z A \ar[d] \ar[r, "\sim"]
                    & \End_{A}\left(D\right) \ar[d] &
                    \delta \otimes \delta' \ar[r, mapsto] & (x \mapsto \delta
                    x \delta') \\
                    D \otimes_Z D^{\opp} \ar[r, "\sim"]
                    & \End_{Z}\left(D\right).
            \end{tikzcd}\]
        \item \label{thm:azumaya-local-7}
            As an $(A \otimes_Z A)$-module $D$ is free of rank $1$. Let $C$ be
            any commutative $Z$-algebra such that $D \otimes_Z C \simeq
            \End_{C}(P)$ for some Zariski locally free $C$-module $P$, then $P$
            viewed as an $(A \otimes_Z C)$-module, is Zariski locally free rank
            $1$.

        \item \label{thm:azumaya-local-8}
            The natural action $D \to \End_{k}(R)$ factors through
            $\End_{R'}(R)$, where $R$ is viewed as an $R'$-module via
            $\varphi_{R/k}$. For any given section $Z \to R'$ of $R' \to Z$, $R$
            is then naturally an $(D \otimes_Z R')$-module. The centraliser and
            double centraliser of $R$ with such a module structure are given by
            \begin{align*}
                Z_{D \otimes_Z R'}(R) \simeq (R')^\opp \simeq R',
                \quad\text{and}\quad
                Z_{D \otimes_Z R'}^2(R) = \End_{R'}(R).
            \end{align*}
            So the induced left $Z_{D \otimes_Z R'}(R)$-module structure on $R$
            is the same as the (right, equivalently the left) $R'$-module
            structure on $R$. Hence $R$ is free as a left $Z_{D \otimes_Z
            R'}(R)$-module of rank $p^d$, in particular, it is finitely
            generated.

        \item \label{thm:azumaya-local-9}
            Given any section $Z \to R'$ of $R' \to Z$, $R$ is a balanced $(D
            \otimes_Z R')$-module, i.e.,
            \[
                D \otimes_Z R' \simeq Z^2_{D \otimes_Z R'}(R) = \End_{R'}(R).
            \]
    \end{enumerate}
\end{theorem}

\begin{proof}
    \begin{enumerate}[wide]
        \item Clearly, $D$ is a free $k$-module with basis $t^{I}\partial^J$,
            $I, J \in \bfZ^d$, with multi-index convention
            \[
                t^I = t_1^{i_1}\cdots t_d^{i_d},\quad \partial^I =
                \partial_1^{i_1} \cdots \partial_d^{i_d},\quad I = (i_1, \ldots,
                i_d) \in \bfZ^d,
            \]
            is used. Note that for every $f \in R$, $1 \leq r \leq p-1$, and $n
            \in \bfN$, we have
            \begin{equation*}
                f \cdot \partial^n_r = \sum_{i = 0}^n (-1)^i \binom{n}{i}
                \partial^{n - i}_r \cdot \partial^i_r(f), \quad
                \partial^n_r \cdot f = \sum_{i = 0}^n \binom{n}{i}
                \partial^i_r(f) \cdot \partial^{n - i}_r.
            \end{equation*}
            So in particular $1 \leq r, r' \leq d$, $n \geq 1$, we have
            \begin{equation}
                [\partial_r^n, t_{r'}] = \delta_{rr'} \cdot n \partial^{n-1}_r,
                \quad\text{and}\quad
                [\partial_{r'}, t_r^n] = \delta_{rr'} \cdot n t^{n-1}_r,
                \label{eq:t-partial-partial-t}
            \end{equation}
            where $\delta_{rr'}$ is the Kronecker delta.

            Clearly, $Z$ is contained in the centre. Let $x := \sum_{I, J} c_{I,
            J} t^I \partial^J \in D$ be arbitrary, $c_{I, J} \in k$. According
            to \eqref{eq:t-partial-partial-t},
            \begin{equation}{}
                [x, t_{r}]
                = \sum_{I, J} c_{I, J} j_r t^I \partial^{J - e_r},
                \quad\text{and}\quad
                [\partial_{r}, x]
                = \sum_{I, J} c_{I, J} i_r t^{I - e_r} \partial^{J},
                \label{eq:t-partial-partial-t-2}
            \end{equation}
            where $e_r$ is the multi-index with $r$-th position $1$ and
            elsewhere $0$. Therefore if $x$ lies in the centre we must have that
            $i_r$ and $j_r$ are multiples of $p$ for all $1 \leq r \leq d$,
            i.e., $x \in Z$. As an $Z$-module,
            \[
                D = \bigoplus_{I, J \in \{0, \ldots,p-1\}^d}Z \cdot
                t^I\partial^J
            \]
            is free of rank $p^{2d}$.

        \item It follows from the same computation as in the first step.

        \item We know that for each $\partial_j$, its $p$-th iteration as a
            derivation is zero. So in this situation, those two maps are given
            by
            \[
                \psi'': \Der_k(R) \otimes_{R, \varphi_R} R \to D,\quad
                \partial_i \otimes t_j \mapsto \iota_R(t_j) \psi(\partial_j)
                = t_j \partial_i^p
            \]
            and
            \[
                \psi': \Der_k(R) \otimes_{R, w} R' \to D,\quad \partial_i
                \otimes t_j' \mapsto
                \iota_R(\varphi_{R/k}(t_j'))\psi(\partial_i) = t_j^p
                \partial_i^p
            \]
            Then we easily see the induced map are ring isomorphisms.

        \item Fix a maximal ideal $\frakm \subseteq Z$. It is clear that the
            centre of $D \otimes_Z \kappa$ is $\kappa$ as is proved in the first
            part.

            For any $1 \leq r \leq p -1$, set $X_r:= (1 \otimes t_r - t_r
            \otimes 1) \in D^e$ and $Y_r = (\partial_r \otimes 1 - 1 \otimes
            \partial_r) \in D^e$. Clearly $X_r \cdot X_s = X_s \cdot X_r$ and
            $Y_r \cdot Y_s = Y_s \cdot Y_r$ for any $1 \leq r, s \leq p -1$. For
            any multi-index $I = (i_1, \ldots, i_d)$, write $X^I := X_1^{i_1}
            \cdot \cdots  \cdot X_d^{i_d}$ and $Y^I := Y_1^{i_1} \cdot \cdots
            \cdot Y_d^{i_d}$ as usual. Then \eqref{eq:t-partial-partial-t-2}
            implies that
            \begin{equation}
                X^J \cdot t^I \partial^J = J! t^I,
                \quad \text{and}\quad
                Y^I \cdot t^I = I!.
                \label{eq:t-acts-as-derivation-on-partial}\\
            \end{equation}

            Let $0 \neq \fraka \subseteq D \otimes_Z \kappa$ be a two sided
            ideal and $0 \neq x = \sum c_{I,J} t^I \partial^J \in \fraka$ with
            $c_{I,J} \in \kappa$. Denote by $\overline X$ and $\overline Y$ the
            image of $X$ and $Y$ in $(D \otimes_Z \kappa)^e$. Let $J_0$ be one
            of the indices such that $|J_0| = \max \{ |J|: c_{I,J} \neq 0\}$.
            Then $\overline X^{J_0} \cdot x = \sum_I J_0! c_{I,J_0} t^I \in
            \fraka$ as $\fraka$ is a two-sided ideal. Let $I_0$ be one of the
            indices such that $| I_0 | = \max \{ |I|: c_{I,J_0} \neq 0\}$. Then
            $\overline{Y}^{I_0} \cdot \overline{X}^{J_0} \cdot x = I_0! J_0!
            c_{I_0,J_0}$, which non-zero since $|I_0|, |J_0| \leq (p-1)!^d$. So
            $1 \in \fraka$. So we conclude that $D \otimes_Z \kappa$ is central
            simple over $\kappa$ and that $D$ is Azumaya over $Z$.

        \item Let $\alpha \in \End_{\bfZ}(D)$ be $(D \otimes_Z A)$-linear. Then
            for any $x \in D$, we have $\alpha(x) = \alpha(x \cdot 1 \cdot 1) =
            x \cdot \alpha(1) \cdot 1 = x\cdot \alpha(1)$. Hence we get a map
            $Z_{D \otimes_Z A}(D) \to D^\opp$, $\alpha \mapsto \alpha(1)$.
            Moreover, for all $a \in A$, we have $\alpha(a) = \alpha(a \cdot 1
            \cdot 1) = \alpha(1 \cdot 1 \cdot \alpha)$, i.e., $a \cdot \alpha(1)
            = \alpha(1) \cdot a$. Hence $\alpha(1) \in Z_D(A) = Z_D^2(R) = A$ by
            the previous result. Conversely, any $A$-linear endomorphism of $D$
            is clearly $(D \otimes_Z A)$-linear. So $Z_{D \otimes_Z A}(D) \simeq
            A$. Then $Z^2_{D \otimes_Z A}(D) = \End_A(D)$ follows.

        \item Now we are going to show that $D \otimes_Z A \to \End_A(D)$,
            $\delta \otimes a \mapsto (x \mapsto \delta x a)$ is an
            isomorphism.
            First observe that both sides, as an $A$-module, is free of rank
            $p^{2d}$. Thus it suffices to show the surjecitvity. Toward this, it
            suffices to show for every maximal ideal $\frakm$ of $A$, the
            induced map
            \[
                D \otimes_Z A \otimes_{A} \kappa \simeq D \otimes_Z \kappa
                \longrightarrow
                \End_{A} \left(D \right) \otimes_{A} \kappa \simeq
                \End_{\kappa} \left(D \otimes_{A} \kappa \right)
            \]
            is surjective, where  $\kappa:= \kappa(m) = A/\frakm$ is the residue
            field of $A$ at $\frakm$. It is easy to see that $D \otimes_Z \kappa
            $ is the quotient of $W_d(\kappa)$ by the two-sided ideal generated
            by $t_r^p - a_r$ and  $\partial_r^p - b_r$, $1 \leq r \leq d$,
            where $a_r$ and $b_r$ are the images of $t^p_r \in Z$ and
            $\partial_r^p \in Z$ in $\kappa$, under $Z \to A \to \kappa$,
            respectively. Similarly, $D \otimes_A$ is the quotient of the
            polynomial ring $\kappa[\partial]$ by the ideal generated
            $\partial_r^p - b_r$, $1 \leq r \leq d$. So in particular, as
            $\kappa$-modules,
            \[
                D \otimes_Z \kappa \simeq \bigoplus_{I, J \in \{ 1, \ldots, p -
                1\}^d } \kappa \cdot t^I \partial^J,
                \text{\quad and \quad}
                D \otimes_A \kappa \simeq \bigoplus\limits_{I, J \in
                \{ 1, \ldots, p-1\}^d} \kappa \cdot \partial^J.
            \]

            We have already know that $D \otimes_A \kappa$ is finitely generated
            over the double centraliser $Z_{D \otimes_Z \kappa }^2(D) = \kappa$
            from the preceding result. So according to \cite[VIII, \S5, \frno3,
            Cor.\ 4]{zbMATH05948094}. It suffices to show that $(D \otimes_{A}
            \kappa)$ is a \emph{simple} left $(D \otimes_Z \kappa)$-module.
            Again denote by $\overline{X}$ the image of $X$ in $D \otimes_Z
            \kappa$. For any $0 \neq x = \sum c_I \partial^I \in D \otimes_A
            \kappa$, $c_I \in \kappa$, denote by $I_0$ one of the indices such
            that $|I_0| = \max \{ |I|: c_I \neq 0 \}$. Then we have
            $\overline{X}^{I_0} \cdot x = I_0!c_{I_0} \neq 0$ thanks to
            \eqref{eq:t-acts-as-derivation-on-partial}. So $(I_0!c_{I_0})\inv
            \overline{X}^{I_0} \cdot x = 1$. Therefore $(D \otimes_Z \kappa)
            \cdot x = D \otimes_A \kappa$. Hence $D \otimes_A \kappa$ is a
            simple $(D \otimes_Z \kappa)$-module.

            Since $Z \to A$ is faithfully flat, we can deduce from the fact that
            $D \otimes_Z A \simeq \End_A(D)$ that $D$ is an Azumaya over $Z$ and
            $D \otimes_Z D^\opp \simeq \End_Z(D)$.

        \item We already know that $X_r \in A \otimes_Z A$. We first claim that
            $\partial^{\underline{\smash{p-1}}} := \partial_1^{p-1} \cdot \cdots
            \cdot \partial_d^{p-1}$ generates $D$ as an $(A \otimes_Z
            A)$-module. Acturally, every $x \in D$ can be written as $x = \sum
            a_I \partial^I$, with $I \in \{1, \ldots, p-1\}^d$. Then because of
            \eqref{eq:t-partial-partial-t-2},
            \begin{equation}
                \sum_I \frac{I!}{(\underline{\smash{p-1}})!}
                (a_I \otimes 1) {X^{\underline{\smash{p-1}}  -I}} \cdot
                \partial^{\underline{\smash{p-1}}} = x.
            \end{equation}
            Therefore, the map $A \otimes_Z A \to D$, given by sending $\alpha
            \in A \otimes_Z A$ to $\alpha \cdot
            \partial^{\underline{\smash{p-1}}}$ is a surjective morphism of $(A
            \otimes_Z A)$-modules. In particular, this morphism reduces to a
            morphism of \emph{left} $A$-modules, whose source and target are
            both free of rank $p^d$. Hence it is an isomorphism.

            Now suppose that we are given a commutative $Z$-algebra $C$ such
            that $D \otimes_Z C \simeq \End_C(P)$ where $P$ is a Zariski locally
            free $c$-module of finite rank. Considering the base change of this
            isomorphism along the map $C \to A \otimes_Z C$, taking notice of
            that $D \otimes_Z A \simeq \End_A(D)$, we obtain a ring isomorphism
            \[
                \End_{A \otimes_Z C}(P \otimes_C (A \otimes_Z C))
                \simeq \End_{A \otimes_Z C}(D \otimes_A (A \otimes_Z C)),
            \]
            where both sides are endomorphism rings of Zarikski locally free $(A
            \otimes_Z C)$-modules. In addition we know that $D \otimes_A (A
            \otimes_Z C))$ is free of rank $1$ as an $(A \otimes_Z A) \otimes_A
            (A \otimes_Z C)$-module, i.e., as an $(A \otimes_Z C)\otimes_C (A
            \otimes_Z C)$-module. So Morita theorem implies that $P \otimes_C (A
            \otimes_Z C)$ is Zariski locally free of rank $1$ as $(A \otimes_Z
            C) \otimes_C (A \otimes_Z C)$. Since $C \to A \otimes_Z C$ is
            faithfully flat and finitely presented, we know $P$ is Zariski
            locally free of rank $1$ as an $(A \otimes_Z C)$-module by descent.

        \item Similarly as in the previous result, every $(D \otimes_Z
            R')$-linear endomorphism $\alpha$ of $R$ is completely determined by
            its image at $1 \in R$ (since $R \subseteq D$), and that this image
            has to be in $k[t^p]$ (since $0 = \alpha(0) = \alpha(x(1)) = x
            (\alpha(1))$, for any $x = \partial^I/I!$) or equivalently, coming
            from $R'$ via $\varphi_{R/k}$.

        \item The proof is similar to that of \ref{thm:azumaya-local-6}. First
            of all, as $R'$-modules, both sides are clearly free of rank
            $p^{2d}$. The strategy is to show that for any maximal ideal
            $\frakm$ of $R'$ with $\kappa:= R'/\frakm$, the reduction $D
            \otimes_Z \kappa \to \End_{\kappa}(R \otimes_{R'} \kappa)$ of the
            natural map $D \otimes_Z R' \to \End_{R'}(R)$ is an isomorphism. To
            show this, one shows that $R \otimes_{R'} \kappa$ is a simple $(D
            \otimes_Z \kappa)$-module; and for this, one observes that for any
            $0 \neq g \in R \otimes_{R'} \kappa \simeq \bigoplus_{I \in \{1,
            \ldots, p-1\}^d} \kappa \cdot t^I$, there is always one $x$ of the
            form $\partial^I/I!$, such that $x(g) = 1$.
    \end{enumerate}
\end{proof}

\section{Global theory: Lie algebroids and modules over them}
\label{sec:D}

Let $f:  X \to S$ be a smooth morphism of schemes. We will interchangeably write
$\cDer(X/S)$ and $\Theta_{X/S}$ for the tangent sheaf.

In this section, we would like to generalise the previous results to a geometric
setting. To us, an \emph{$(f\inv\scrO_S, \scrO_X)$-Lie algebra}, usually called
a \emph{Lie algebroid}, is a quasi-coherent $\scrO_X$-modules $L$, together with
maps $[-,-]:L \times L \to L$, and $\rho: L \to \cDer(X/S)$, such that locally
over $U \subseteq X$, $(L(U), [-, -], \rho)$ is an $\left((f\inv\scrO_S)(U),
\scrO_X(U)\right)$-Lie algebra. Moreover, in case $S$ is of characteristic $p >
0$, i.e., $p\cdot \scrO_S = 0$, we similarly have the notion of \emph{restricted
$(f\inv \scrO_S, \scrO_X)$-Lie algebra} or \emph{restricted Lie algebroid}, $(L,
[-,-], -^{[p]}, \rho)$, which Zariski locally over $U \subseteq X$ is a
restricted $((f\inv \scrO_S)(U), \scrO_X(U))$-Lie algebra. For any section
$\lambda \in \scrO_S(S)$, we can similarly define $\leftidx{^\lambda}L$ (see
Example~\ref{exa:restricted-k-R-Lie}.\ref{exa:restricted-k-R-Lie-2}). Moreover,
the construction of the universal enveloping algebra glues. Hence we obtain a
sheaf $U_{f\inv \scrO_S, \scrO_X}(L)$ of non-commutative rings, satisfying the
universal property as one may expect.

We can also talk about $L$-modules as in the local case. Precisely,
an \emph{$L$-module} structure on a quasi-coherent $\scrO_X$-module $E$ is a
morphism $\sigma_L: L \to \End_{f\inv\scrO_S}(E)$, which is a morphism of left
$\scrO_X$-modules as well as a morphism of $(f\inv\scrO_S)$-Lie algebras, and
moreover Zariski locally \eqref{eq:compatible} holds. Then
Corollary~\ref{cor:universal-property-of-U(L)} implies that the category of
$L$-modules is equivalent to the category of left $U_{f\inv\scrO_S,
\scrO_X}(L)$-modules. In particular, we are interested in $L$-modules that are
finite locally free as $\scrO_X$-modules. Denote by  $\catVect_{X/S, r;L} \to
(\catSch/S)$ the stack of such locally free sheaves. Moreover, we write
$\catLoc_{X/S, r}:= \catVect_{X/S, r; \cDer(X/S)}$, whose objects are called
\emph{local systems} of rank $r$ on $X/S$.

\begin{example}\label{exa:typical-example}
    We always have the following two cases in mind.
    \begin{enumerate}
        \item $L:= K^\vee:=\cHom(K, \scrO_X)$ is a \emph{trivial} Lie algebroid,
            i.e., with zero bracket, zero anchor (and zero $p$-operator if in
            characteristic $p$), where $K$ is a finite locally free
            $\scrO_X$-module.

        \item $L := \leftidx{^\lambda}\cDer(X/S)$ is the (restricted, if in
            characteristic $p$) tangent Lie algebroid twisted by a section
            $\lambda \in \scrO_S(S)$. Most importantly, we will study the cases
            where $\lambda = 0$ and $\lambda = 1$
    \end{enumerate}
    In the first case, $U_{f\inv\scrO_S, \scrO_X}(K^\vee) = \Sym^\bullet K^\vee$
    is the symmetric algebra according to
    Example~\ref{exa:k-R-Lie-Mod-and-U}.\ref{exa:k-R-Lie-Mod-and-U-1}. In the
    second case, $\leftidx{^\lambda}D_{X/S}:= U_{f\inv \scrO_S,
    \scrO_X}\left(\leftidx{^\lambda}\cDer(X/S)\right)$ is the sheaf of
    $\lambda$-twisted \emph{crystalline differential operators}.%
    \footnote{
        We take this as our definition of the sheaf of differential operators.
        In case $\lambda = 1$, this sheaf is different from the one defined in
        \cite[\S16]{EGAIV4}, but coincide with the one defined in
        \cite{zbMATH03467279}, where it is called the sheaf of divided power
        differential operators. Sometimes it is also called the sheaf of
        differential operators without divided powers.
    }
    We will simply write $D_{X/S}$ for $\leftidx{^1}D_{X/S}$, and of course
    $\leftidx{^0}D_{X/S}= \Sym^\bullet \Theta_{X/S}$.

    Recall that a \emph{$\lambda$-connection} on $E$ is a morphism $\nabla: E
    \to E \otimes \Omega_{X/S}^1$ of $(f\inv\scrO_S)$-modules, such that for
    local sections $f$ of $\scrO_X$ and $e$ of $E$, $\nabla(fe) = \lambda e
    \otimes \rmd f + f\nabla(e)$. Such a connection extends to a homomorphism
    $\nabla: E\otimes \Omega_{X/S}^1 \to E \otimes \Omega_{X/S}^2$, $e \otimes
    \omega \mapsto \lambda e \otimes \rmd \omega + \nabla(e) \wedge \omega$ for
    local sections $e$ of $E$ and $\omega$ of $\Omega_{X/S}^1$. A
    $\lambda$-connection is said to be \emph{flat} or \emph{integrable} if the
    composition $E \to E \otimes \Omega_{X/S}^1 \to E \otimes \Omega_{X/S}^2$ is
    zero. Conventionally, a flat $1$-connection is simply called a \emph{flat
    connection} and a flat $0$-connection is called a \emph{Higgs field}.

    It is easy to see that a flat $\lambda$-connection is the same as an
    $\leftidx{^\lambda}\cDer(X/S)$-module. Actually, $\Omega_{X/S}^1$ is locally
    free, a $\lambda$-connection $\nabla$ gives rise to an $\scrO_X$-morphism
    $\leftidx{^\lambda}\cDer(X/S) \to \cEnd_{f\inv\scrO_S}(E)$. The flatness
    condition implies that this map is a morphism of $(f\inv\scrO_S)$-Lie
    algebras. One then check this makes $E$ and
    $\leftidx{^\lambda}\cDer(X/S)$-module. Then by universal property of the
    universal enveloping algebra flat $\lambda$-connections are also the same as
    $\leftidx{^\lambda}D_{X/S}$-modules.
\end{example}

\subsection{$K$-Higgs Bundles}
Fix a locally free sheaf $K$ of rank $\rk(K)>0$. Equip $L := K^\vee$ with the
trivial Lie algebroid structure. Denote by $\pi$ the canonical projection
$\bfV(K) := \cSpec_X \Sym^\bullet L \to X$. For example, we may take
$K=\Omega_{X/S}^1$ in case $X/S$ is of relative dimension $d$. Then $\bfV(K) =
\cotsp{X/S}$ is the cotangent bundle of $X/S$, and $L$ is
$\leftidx{^0}\cDer(X/S)$.

A \emph{$K$-Higgs filed} (or an \emph{$L$-co-Higgs field}) on a quasi-coherent
$\scrO_X$-module $E$ is just
\begin{enumerate}
    \item a $(\Sym K^\vee)$-module structure on $E$; or equivalently,
    \item an isomorphism $E \simeq \pi_\ast \widetilde{E}$ for some
        quasi-coherent $\scrO_{\bfV(K)}$-module; or equivalently,
    \item a morphism $\sigma: K^\vee \to \cEnd_{\scrO_X}(E)$ of
        $\scrO_X$-modules, such that $\sigma(x)\circ\sigma(y)
        -\sigma(y)\circ\sigma(x)= 0 \in \cEnd_{\scrO_X}(E)(U)$ for any local
        sections $x$ and $y$ of $K^\vee$ over $U \subseteq X$; or equivalently,
    \item a morphism $\theta: E \to E \otimes K$ of $\scrO_X$-modules such
        that the composition $E \to E \otimes K \xrightarrow{\theta \otimes \id}
        E \otimes K \otimes K \to E \otimes \wedge^2 K$ is zero.
\end{enumerate}

A \emph{$K$-Higgs module} $(E, \theta)$ is an quasi-coherent $\scrO_X$-module
$E$ with a $K$-Higgs field $\theta$. If in addition $E$ is locally free of rank
$r>0$, we call it a \emph{$K$-Higgs bundle} of rank $r$.
If $K = \Omega_{X/S}^1$, we just call it a Higgs
module/bundle. Denote by $\catHiggs_{X/S,r;K} \to (\catSch/S)$ the stack of
$K$-Higgs bundles on $X/S$ of rank $r$. In other words, $\catHiggs_{X/S, r;K} :=
\catVect_{X/S, r; \leftidx{^0}L}$.
Finally, set $\catHiggs_{X/S,r} := \catHiggs_{X/S, r; \Omega_{X/S}^1}$.

\subsection{The (generalised) $p$-curvature of a restrictired Lie algebroid
module}\label{ssec:p-curv}
Suppose that $S$ is of characteristic $p>0$. Let $L:= (L,[-,-], -^{[p]}, \rho)$
be a \emph{restricted} Lie algebroid and let $U(L):= U_{f\inv \scrO_S,
\scrO_X}(L)$ be its enveloping algebra with natural map $\iota:L \to U(L)$.
Suppose moreover that $L$ is finite locally free.%
\footnote{
    This assumption is not essential but makes the exposition simpler.
}
Let $K:=L^\vee:=\cHom(L, \scrO_X)$ be its $\scrO_X$-module dual.
Then Proposition~\ref{cor:p-linear-universal} implies that we have a morphism of
sheaves of abelian groups
\begin{equation}
    \psi: L \longrightarrow U(L),
    \quad \psi(x) = (\iota(x))^{p} - \iota(x^{[p]})
    \label{eq:p-curv-genuine}
\end{equation}
for any local section $x$ of $L$. This morphism, according to
Corollary~\ref{cor:p-linear-universal}, has the following properties.
\begin{enumerate}[wide]
    \item It is $\Fr_X$-linear, hence defines a morphism
        \begin{equation*}
            \Fr_X^\ast L \longrightarrow U(L)
        \end{equation*}
        of sheaves (left) $\scrO_X$-modules; This map has image in the
        centraliser of $\scrO_X$ inside $U(L)$. Moreover,
        this centraliser is a sheaf of \emph{commutative} rings. Hence
        we obtain a map
        \begin{equation}
            \Sym^\bullet \left(\Fr_X^\ast L\right) \longrightarrow
            Z_{U(L)}(\scrO_X) \subseteq U(L) .
            \label{eq:p-curvature-restriction}
        \end{equation}

    \item Adjointly ($F_{X/S}^\ast \dashv F_{X/S, \ast}$), we have a map of
        $\scrO_{X'}$-modules
        \begin{equation*}
            w^\ast L \longrightarrow F_{X/S, \ast} U(L).
        \end{equation*}
        The image of this map lies in the centre of $F_{X/S, \ast}U(L)$. Hence
        it further induces a map
        \begin{equation}
            \begin{tikzcd}[column sep = small]
                \Sym^\bullet (w^\ast L) \ar[r, "\psi"] &
                Z\left(F_{X/S,\ast} U(L)\right) \subseteq F_{X/S, \ast} U(L),
            \end{tikzcd}
            \label{eq:p-curv}
        \end{equation}
        where $Z$ stands for taking the centre of a sheaf of non-commutative
        rings.
\end{enumerate}

Therefore, given a $U(L)$-module structure $\nabla$ on a quasi-coherent
$\scrO_X$-module $E$, by restricting the action of $U(L)$ on $E$ to
$\Sym^\bullet (\Fr_X^\ast L)$ via \eqref{eq:p-curvature-restriction}, we obtain
an $(\Fr_X^\ast K)$-Higgs field on $E$. This $(\Fr_X^\ast K)$-Higgs field is
called the \emph{$p$-curvature} of $\nabla$. Similarly, the restriction of the
action of $F_{X/S, \ast} U(L)$ on $F_{X/S,\ast}E$ to $\Sym^\bullet (w^\ast L)$
via \eqref{eq:p-curv}, gives an $(w^\ast K)$-Higgs field on $F_{X/S, \ast} E$,
which is locally free of rank $p^d$ in case $X/S$ is smooth of relative
dimension $d$. In this way, we obtain morphisms
\begin{align}
    \catVect_{X/S, r;L} & \longrightarrow \catHiggs_{X/S, r; \Fr_X^\ast
    K}, \label{eq:K-p-curv} \\
    \text{and,\quad}
    \catVect_{X/S, r;L} &
    \longrightarrow \catHiggs_{X/S, p^d r; w^\ast K}, \label{eq:K-p-curv'}
\end{align}
of stacks over $S$.

Moreover, \eqref{eq:p-curv} realises $F_{X/S, \ast}U(L)$ a quasicoherent sheaf
on $\bfV(w^\ast K)$, which we denoted by $\scrU(L)$.

\begin{example}
    In case $L = \cDer(X/S)$, we obtain the usual notion of $p$-curvature for a
    flat conneciton, cf., \cite[\S5]{zbMATH03350023}. In case $L =
    \leftidx{^\lambda}\cDer(X/S)$, where $\lambda \in \scrO_S(S)$, we obtain the
    $p$-curvature defined by \cite[Definition~3.1]{zbMATH01588211}.
\end{example}

\subsection{The Hitchin morphism for $K$-Higgs bundles}
Let $f: X \to S$ be a \emph{smooth} scheme of relative dimension $d$. So
$\Omega_{X/S}^1$ is a locally free $\scrO_X$-module of rank $d$.

Let $E$ be a $K$-Higgs bundle of rank $r$, with $K$-Higgs field $\theta: E \to E
\otimes K$. Such a Higgs field $\theta$ is a \emph{$K$-twisted endomorphism} of
$E$ in the sense of~\S\ref{ssec:characteristic-polynomial}. Therefore, we have
the characteristic polynomial
\[
    \chi_\theta(t) = t^r - a_1 t^{r - 1} + \cdots + (-1)^{r}a_r \in
    \Gamma\big(X, (\Sym^\bullet K)[t]\big),
    \quad a_i \in \Gamma(X, \Sym^i K),
\]
of $\theta$. Define on the site $(\catSch/S)_{\fppf}$ the presheaf
\[\begin{aligned}
        \bfB_{X/S, r; K}:
        (\catSch/S)^{\text{op}} \longrightarrow \catSet, \quad
        (T \to S) \longmapsto
        \bigoplus_{i=1}^r \Gamma\big(X_T, \Sym^i K_T \big),
\end{aligned}\]
where $X_T:= X \times_{S} T$ and $K_T$ is the pullback of $K$ to $X_T$. This is
called the \emph{Hitchin base}. We simply write $\bfB_{X/S, r}$ for $\bfB_{X/S,
r; \Omega_{X/S}^1}$. We identify a $T$-point $\chi = (a_1, \ldots,
a_r) \in \bfB_{X/S, r; K}(T)$ as the polynomial
\begin{equation*}
    \chi(t) = t^r + \cdots + (-1)^i a_i t^{r - i} + \cdots +
    (-1)^r a_r \in \Gamma\big(X_T, (\Sym^\bullet K_T)[t]\big)
\end{equation*}
with $a_i \in \Gamma\big(X_T,  \Sym^i K_T \big)$.
With this definition, we obtain a morphism
\begin{equation}
    c_{\Higgs}:=c_{\Higgs;K}:
    \catHiggs_{X/S, r;K} \longrightarrow \bfB_{X/S, r;K}
    \label{eq:c-dol}
\end{equation}
of stacks over $S$, which sends a $K$-Higgs bundle to (the coefficients of) the
characteristic polynomial of its $K$-Higgs field.

The following Propositoin is a slight generalisation of
\cite[Prop.~3.2]{zbMATH01588211}, \cite[Def.~3.16]{zbMATH06666980}, and
\cite[Prop.~3.1]{zbMATH06526258}.

\begin{proposition}\label{prop:reduced-char-polyn}
    Let $X/S$ be a smooth morphism of relative dimension $d$ and suppose that
    $S$ is of characteristic $p$. Let $L$ be a restricted Lie algebroid on $X/S$
    such that the underlying $\scrO_X$-module is finite locally free. Set
    $K:=\cHom(L, \scrO_X)$. Fix an integer $r >0$. Then there is a map
    \begin{align*}
        c_{\dR}:\catVect_{X/S, r; L} \to \bfB_{X'/S, r;w^\ast K},
    \end{align*}
    of stacks, rendering the diagram
    \[
        \begin{tikzcd}[column sep = tiny]
            & \catHiggs_{X'/S, p^d r; w^\ast K}
            \ar[rr, "\text{\eqref{eq:c-dol}}"', "c_{\Higgs, w^\ast K}"]
            &[-4.5em]& \bfB_{X'/S, p^d r; w^\ast K}
            \ar[dr, "F_{X/S}^\ast"] \\
            \catVect_{X/S,r;L}
            \ar[dr, "\text{\eqref{eq:K-p-curv}}"']
            \ar[rr, dotted, "{\exists\, c_{\dR}}"]
            \ar[ur, "\text{\eqref{eq:K-p-curv'}}"]
            && \bfB_{X'/S,r; w^\ast K}
            \ar[dr, "F_{X/S}^\ast"]
            \ar[ur, "{(-)}^{p^d}"'] 
            &&[-1.5em] \bfB_{X/S, p^d r; \Fr_X^\ast K} \\
            & \catHiggs_{X/S, r;\Fr_X^\ast K}
            \ar[rr, "c_{\Higgs; \Fr_X^\ast K}", "\text{\eqref{eq:c-dol}}"']
            && \bfB_{X/S, r; \Fr_X^\ast K}
            \ar[ur, "{(-)}^{p^d}"']
         \end{tikzcd}
     \]
     commutative.
 \end{proposition}

\begin{proof}
    \begin{enumerate}[wide]
        \item (The parallelogram on the right commutes) This is obvious.
        \item (Functoriallity) The following arguments works for all $T$-points.
            In particular, we remark that faithful flatness is preserved under
            arbitrary base change. Hence $X_T/T$ is faithfully flat for all
            $T/S$.
        \item (The outer hexagon commutes)
            Given a rank $r$ locally free sheaf $E$ with an $L$-module structure
            on $X/S$, we know its $p$-curvature $\psi$ is a $(\Fr_X^\ast
            K)$-Higgs field over $E$, which is of rank $r$ (resp.\ $\psi'$ is a
            $(w^\ast K)$-Higgs field on $F_{X/S, \ast}E$, which has rank $p^d
            r$). Denote by $\chi$ and $\chi'$ the characteristic polynomials of
            $\psi$ and $\psi'$ respectively. Then we have
            \begin{equation}
                F^\ast_{X/S}\chi' = \chi^{p^d}
                \in \bfB_{X/S, p^d r; \Fr_X^\ast K}(S),
                \label{eq:descent-1}
            \end{equation}
            i.e., the outer hexagon of the diagram commutes. In fact, the
            characteristic polynomial $\chi$ is that of the $(F_{X/S}^\ast
            (w^\ast K))$-module $E \otimes F_{X/S}^\ast \Sym^\bullet w^\ast K$.
            Meanwhile, the characteristic polynomial $\chi'$ is that of the
            $(\Sym (w^\ast K))$-module $F_{X/S, \ast}E \otimes \Sym^\bullet
            (w^\ast K) \simeq F_{X/S,\ast}(E \otimes F^\ast_{X/S}\Sym^\bullet
            (w^\ast K) )$ (using projection formula for each degree then taking
            direct sum), via the natural map
            \[
                \Sym^\bullet (w^\ast ) \longrightarrow
                F_{X/S, \ast} F^\ast_{X/S} \Sym^\bullet (w^\ast K).
            \]
            Then the equality follows from
            Proposition~\ref{lem:norm-of-Frobenius} and \cite[Chaptitre~III,
            \S9, \frno4, Propostion~6]{zbMATH05069335} and the fact that $\Nm$
            respects arbitrary base change
            (\cite[\SPtag{0BD2}]{stacks-project}).

        \item \label{item:cdr-exists}
           (The dotted arrow exists, making the lower left parallelogram
           commute) The fact that $\chi$ is a pullback of some $\chi'$ by
           $\Fr_{X/S}^\ast$ follows from Lemma~\ref{lem:char-descent} by gluing.

            More precisely, by the very definition, the characteristic
            polynomial $ \chi(t) = t^r - c_1 t^{r-1} + \cdots + (-1)^n c_n$ of $
            \psi : E \to E \otimes \Fr_X^\ast K$ has coefficients
            \[
                c_i \in \rmH^0\left(X, \Sym^i \Fr_X^\ast K\right) \simeq
                \rmH^0\left(X, \Fr_X^\ast \Sym^i K\right).
            \]
            By definition/construction, $c_i$ respects restriction. That is to
            say, if the characteristic polynomial of $E \to E \otimes \Fr_X^\ast
            K$ has coefficients $c_i \in \rmH^0(X, \Sym^i \Fr_X^\ast K)$, then
            for any open $U \subseteq X$, the characteristic polynomial of $E|_U
            \to E|_U \otimes_{\scrO_U} (\Fr_X^\ast K)|_U$ has coefficients the
            image of $c_i$ under the restriction map $\rmH^0(X, \Sym^i
            \Fr_X^\ast K) \to \rmH^0(U, \Sym^i \Fr_X^\ast K)$.

            Take an affine open cover $\frakV$ of $S$ then take an affine open
            cover $\frakU$ of $X$ such that each affine open $U$ in $\frakU$ is
            mapped into some affine open $V$ in $\frakV$, and that over $U$, $K$
            and $E$ are both trivialised. Set $\frakU':= w^\ast \frakU$, and
            $U'= w\inv (U)$ for each $U \in \frakU$. Then $\frakU'$ is an affine
            open cover for $X'$ and over each $U'$, $w^\ast K$ is trivialised.

            We know from Lemma~\ref{lem:char-descent} that over each $U$ in
            $\frakU$, $ c_i|_U = \Fr_X^\ast \gamma_U = F_{X/S}^\ast \gamma'_U$,
            for some $\gamma_U \in \Gamma(U, K)$ and $\gamma_U' = w^\ast
            \gamma_U \in \Gamma(U', w^\ast K)$. Moreover, we can conclude from
            $(c_i|_{U_1})|_{U_1 \cap U_2} = c_i|_{U_1 \cap U_2} = (c_i|_{U_2})
            |_{U_1 \cap U_2}$ that $(F_{X/S}^\ast \gamma_{U_1}')|_{U_1 \cap U_2}
            = F_{X/S}^\ast (\gamma'_{U_1}|_{U_1'\cap U_2'}) =  F_{X/S}^\ast
            (\gamma'|_{U_1 \cap U_2}) = F_{X/S}^\ast (\gamma_{U_2}'|_{U_1' \cap
            U_2'})$ for $U_1$ and $U_2$ in $\frakU$. Since $F_{X/S}$ is
            faithfully flat, we know that the pullback map $K \to F_{X/S, \ast}
            F^\ast_{X/S} K$ is \emph{injective} (\cite[{}2.2.8, 2.2.9]{EGAIV2}).
            Therefore, we further conclude that $\gamma_{U_1}'|_{U_1' \cap
            U_2'}\gamma_{U_2}'|_{U_1' \cap U_2'}$. Hence these $\gamma'_U$'s, $U
            \in \frakU$, glue to a section $\gamma_i \in \Gamma(X', w^\ast K)$,
            and $F_{X/S}^\ast \gamma'_i = c_i$. All these $\gamma_i$ defines a
            $\chi'' \in \bfB_{X/S, r; w^\ast K}(S)$, such that
            \begin{equation}
                F_{X/S}^\ast \chi '' = \chi \in \bfB_{X/S, r; w^\ast K}(S)
                \label{eq:descent-2}
            \end{equation}

        \item (The upper left parallelogram commutes)
            Combining \eqref{eq:descent-1} and \eqref{eq:descent-2}, and again
            the fact that $F_{X/S}^\ast : w^\ast K \to F_{X/S, \ast}
            F_{X/S}^\ast (w^\ast K)$ is injective, we obtain that
            \begin{equation}
                \chi' = (\chi'')^{p^d} \in \bfB_{X'/S,p^d r;w^\ast K}(S).
                \label{eq:reduced-char-polyn}
            \end{equation}
    \end{enumerate}
    So we complete the proof.
\end{proof}

\begin{remark}
    \begin{enumerate}[wide]
        \item Suppose that $X$ is locally noetherian and regular, for example it
            is smooth over $\Spec k$ for some field $k$, then according to
            Kunz's theorem (\citeSP{0EC0}), $\Fr_X$ is faithfully flat.
            Therefore, in the above step~\ref{item:cdr-exists}, we can actually
            glue $\gamma_U$'s to get a global section in $\Gamma(X, K)$ such
            that $c_i$ is its pullback under $\Fr_X$. However, regularity is
            \emph{not} preserved by base change. So this procedure does not work
            functorially. Hence we \emph{cannot} produce a map of stacks
            $\catVect_{X/S, r; L} \to \bfB_{X/S, r ;K}$, but only for those $T$
            such that $X_T$ is regular, a map $\catVect_{X/S, r;L}(T) \to
            \bfB_{X/S, r; K}(T)$.
        \item The existence of $c_{\dR}$ is usually proved by Cartier descent
            (\cite[Thm.~5.1]{zbMATH03350023}), using the fact that the
            $p$-curvature is horizontal with respect to the canonical
            connection. Our proof here avoids using it. Instead the result just
            follows from an easy observation Lemma~\ref{lem:char-descent}.
    \end{enumerate}
\end{remark}

\subsection{The Spectral cover defined by a $K$-Higgs bundle}
\label{ssec:spectral-cover}
Continue with the same setting as in the previous subsection. Recall that there
is a global section $\lambda \in \Gamma(\bfV(K), \pi^\ast K)$, corresponding to
$\id \in \End_{X}(\bfV(K))$, under the
identification~\eqref{eq:section-of-vector-bundles}, which is usually called the
\emph{tautological section}. Then the determinant $\det(\lambda -
\pi^\ast\theta)=: \chi_\theta(\lambda)$ is a global section of the locally free
sheaf $\Sym^r_{\scrO_{ \bfV(K) }} (\pi^\ast K) = \pi^\ast \Sym^r_{\scrO_X} K$ of
rank $S(r, \rk(K))$ over $\bfV(K)$, which is obtained from the
characteristic polynomial $\chi_\theta(t)$ of $\theta$ by pulling back along
$\pi$ then substituting $t$ by the tautological section $\lambda$, as described
in the general setting in~\S\ref{ssec:characteristic-polynomial}. The vanishing
locus on $\bfV(K)$ of the section $\chi_\theta(\lambda): \scrO_{\bfV(K)} \to
\pi^\ast \Sym^r K$ is defined by the quasi-coherent sheaf of ideals
\begin{equation}
    I_{\theta} := I_{\chi_\theta} := \Image\left(
        \begin{tikzcd}[column sep = large]
            {\left(\pi^\ast(\Sym^r K )\right)}^\vee
            \ar[r, "{(\chi_\theta(\lambda))}^\vee"]
            & \scrO_{\bfV(K)}
        \end{tikzcd}
    \right)
    \subseteq \scrO_{\bfV(K)}.
    \label{eq:ideal-of-definition}
\end{equation}
Denote the corresponding closed embedding by
$\iota_{\theta}:=\iota_{\chi_\theta}: Z_{\chi_\theta} \to \bfV(K)$. Besides, the
morphism $\pi \circ \iota_{\theta}: Z_{\chi_\theta} \to X$, or just
$Z_{\chi_\theta}$, is called the \emph{spectral cover} of $X$ associated to $(E,
\theta)$. Proposition~\ref{prop:proper-of-universal-spectral-curve} justifies
its name.

Given any $T$-point $\chi$ of $\bfB_{X/S, r;K}$, So by pulling back this
polynomial to $\bfV(K)_T$ and substituting $t$ by the tautological section
$\lambda$, we can obtain a closed subscheme $Z_\chi$ in the same manner. Denote
the ideal sheaf of definition ob $Z_\chi$ by $I_\chi$.

In case $\bfB_{X/S, r; K}$ is representable by a scheme $B$, we then have an
\emph{universal} element $\chi_{\text{univ}} \in \bigoplus_{i=1}^r
\Gamma\left(X_B, \Sym^i K_B \right)$ corresponding to $\id_B$. So there is a
corresponding closed subscheme $Z:=Z_{\text{univ}}$ of $\bfV(K_B) = \bfV(K)_B$.
We will call $Z/X_B$ the \emph{universal spectral cover}. It is universal in the
sense that for any $\chi: T \to B$ over $S$, $Z_\chi$ is the pullback of
$Z_{\text{univ}}$ along $\chi$, which follows just from Yoneda.

\begin{example} \label{exa:proper-over-k-representable}
    In case $X/S = X/k$ is proper over a field $k$,
    $\Gamma(X, K)$ is a finite dimensional $k$-vector space,
    and any map $T \to \Spec k$ is flat. So flat base change implies that
    $\bfB_{X/S, r;K}$ is representable by the affine $k$-scheme
    \[
        B_{X/S,r;K}:=\Spec \Sym^\bullet
        \left(\bigoplus_{i=1}^r (\Gamma(X, \Sym^i K)\right)^\vee
    \]
    (see~\S\ref{ssec:affine-morphism} and cf.~\cite[20]{zbMATH00753832}).
\end{example}

\begin{example}\label{exa:computation}
    Here we give an explicit \emph{local equation} for $Z_\chi$. Simply write
    $\bfB$ for $\bfB_{X/S, r;K}$. To this aim, we may assume that $\chi \in
    \bfB(S)$ is an $S$-point of $\bfB$, by observing all the following arguments
    works functorially. Assume that $K$ is of rank $d$ (not necessarily equal to
    the relative dimension of $X/S$). Moreover, we may assume that $X = \Spec R$
    is an affine scheme and the sheaf $K$ is free; otherwise, replace $X$ by a
    small enough affine open subscheme $U= \Spec R$ over which $K|_U$ is
    trivialised $K \simeq \bigoplus_{i = 1}^d \scrO_{X} \cdot \omega_i$,
    where $\omega_i \in \Gamma(X, K)$, $1 \leq i \leq d$, form a basis for the
    free $R$-module $\Gamma(X, K)$. This fixed trivialisation induces
    isomorphisms
    \begin{align*}
        \pi^\ast K & \simeq \bigoplus_{i = 1}^d
        \scrO_{\bfV(K)} \cdot \pi^\ast \omega_i,\\
        \text{\quad and\quad}
        \pi_\ast \scrO_{\bfV(K)}
        & \simeq \Sym^\bullet K^\vee \simeq
        \scrO_{X}[\partial_1, \ldots, \partial_d],
    \end{align*}
    with $\partial_i \in {\Gamma(X, K)}^\vee = \Gamma(X, K^\vee) \subseteq
    \Gamma(\bfV(K), \scrO_{\bfV(K)})$ being the $R$-dual of $\omega_i$.
    Moreover, for any  $1 \leq m \leq r$,
    \begin{equation}
        \begin{aligned}
            \Sym^m K
            & \simeq \bigoplus_{ |\underline{i}| = S(m, d) }
            \scrO_{X} \cdot \underline{\omega}^{\underline{i}},
            \text{\quad and\quad}\\
            \pi^\ast\Sym^m K
            & \simeq \bigoplus_{ |\underline{i}| = S(m, d) }
            \scrO_{\bfV(K)} \cdot
            \pi^\ast\underline{\omega}^{\underline{i}},
            \label{eq:decomposition-of-Sym-Omega}
        \end{aligned}
    \end{equation}
    where the usual multi-index convention is used; that is, for any multi-index
    $\underline{i} = (i_1, i_2, \ldots, i_d)$, where $i_j \geq 0$ for all $1
    \leq j \leq d$, we set $|\underline{i}| := \sum_{j = 1}^d i_j$ and
    $\underline{\omega}^{\underline{i}} := \omega_1^{i_1} \cdot \omega_2^{i_2}
    \cdots \omega_d^{i_d} \in \Gamma(X, \Sym^r K)$. The tautological section in
    this case can be written as
    \begin{equation}
        \lambda =\sum_{j=1}^d \partial_j \cdot (\pi^\ast \omega_j) \in
        \Gamma(\bfV(K), \pi^\ast K).
        \label{eq:tautological-section}
    \end{equation}

    For any $1 \leq m \leq r$, and any $a_m \in \Gamma(X, \Sym^m K)$,
    we can write, according to \eqref{eq:decomposition-of-Sym-Omega},%
    \[
        a_m = a_{m1} \cdot \omega_1^m + a_{m2} \cdot \omega_2^m +
        \cdots + a_{md} \cdot \omega_d^m
        + (\text{other terms}),
    \]
    where $a_{mj} \in \Gamma(X, \scrO_{X}) = R$, for each $1\leq j \leq d$, and
    $\text{(other terms)}$ is a sum of terms of the form $a_{m\underline{i}}
    \cdot \underline{\omega}^{\underline{i}}$, with $a_{m\underline{i}}\in R$
    and with $|\underline{i}| \neq 1$. Similarly, the equation
    \eqref{eq:tautological-section} gives that
    \[
        \lambda^{r - m} = \partial_1^{r - m} \cdot \pi^\ast \omega_1^{r-m} +
        \partial_2^{r - m} \cdot \pi^\ast \omega_2^{r-m} + \cdots +
        \partial_d^{r - m} \cdot \pi^\ast \omega_d^{r-m} + (\text{other terms}).
    \]
    Therefore, any $\chi$ of the form
    $\lambda^r - a_1 \lambda^{r-1} + \cdots + (-1)^r a_r$ of
    $\pi^\ast \Sym^r K$ can be then written as
    \begin{equation}
        \begin{aligned}
            \chi & = \lambda^r - a_1\lambda^{r -1} + \cdots + (-1) a_r \\
            & = \phantom{+}
            (\partial_1^r - a_{11}\partial_1^{r-1} + \cdots + (-1)^r a_{r1})
                \cdot \pi^\ast \omega_1^r \\
            & \phantom{=} +
                (\partial_2^r - a_{12}\partial_1^{r-1} + \cdots + (-1)^r a_{r2})
                \cdot \pi^\ast \omega_2^r\\
            & \phantom{=} + \cdots \\
            & \phantom{=} +
                (\partial_d^r - a_{1d}\partial_1^{r-1} + \cdots + (-1)^r a_{rd})
                \cdot \pi^\ast \omega_d^r\\
            & \phantom{=} + (\text{other terms})\\
            & =: g_1 \cdot \pi^\ast \omega_1^r
                + g_2 \cdot \pi^\ast \omega_2^r
                + \cdots
                + g_d \cdot \pi^\ast \omega_d^r
                + (\text{other terms}). %
        \end{aligned}
        \label{eq:local-defining-equation}
    \end{equation}
    In the above equation, $\text{(other terms)}$ is a sum of terms of the form
    $g_{\underline{i}} \cdot \pi^\ast \underline{\omega}^{\underline{i}}$,
    where $g_{\underline{i}} \in R[\partial_1,\ldots,\partial_d]$
    are polynomials of degree $r$, and for each $1 \leq m \leq d$,
    we write $g_m$ instead of $g_{\underline{i}}$
    when $|\underline{i}| = 1$ and $i_m = 1$.

    So we can conclude that $Z_\chi$ is (Zariski locally over $\Spec R \subseteq
    X$defined by the ideal $(g_1, \ldots, g_d, \ldots, g_{\underline{i}},
    \ldots)$ of $R[\partial_1, \ldots, \partial_d]$, generated by $S(r, d)$
    polynomials; in other words, $Z_\chi$ is the spectrum of
    \begin{equation}
        \frac{R[\partial_1, \ldots, \partial_d]}{
            (g_1, \ldots, g_d, \ldots,
            g_{\underline{i}}, \ldots)
        }
        \label{eq:Z-chi}
    \end{equation}
    We remark that $g_m \in R[\partial_m]$ with $1 \leq m \leq d$ are
    polynomials in only one variable of degree $r$.
\end{example}

\begin{proposition} \label{prop:proper-of-universal-spectral-curve}
    For any $\chi\in \bfB_{X/S, r;K}(T)$ with \emph{non-empty} $Z_\chi$, the
    natural map $Z_\chi \to X_T$ is finite and locally of finite presentation,
    hence in particular is proper. Suppose that $X/S$ is proper, and that
    $\bfB_{X/S, r;K}$ is representable by a scheme $B$, then $Z/B$ is proper,
    where $Z/X_B$ is the universal spectral cover.
    Moreover, if $r = 1$ or $d = 1$, $Z_\chi$ is always non-empty and
    $Z_\chi/X_T$ is also flat so finite locally free. Hence if in addition $X/S$
    is flat, so is $Z/B$.
\end{proposition}

\begin{proof}
    Finiteness is a local question. According to Example~\ref{exa:computation},
    for any affine open $U = \Spec R \subseteq X_T$ where $K_T$ is trivialised,
    $Z_\chi|_U$ is the spectrum of the $R$-algebra given in \eqref{eq:Z-chi}.
    One observe that for each $1 \leq m \leq d$, $g_m$ has a single variable
    $\partial_m$. Therefore, \eqref{eq:Z-chi} is a finite $R$-module, and
    $Z_\chi/X_T$ is finite. In case $X/S$ is proper and $\bfB_{X/S, r;K}$ is
    representable by a scheme $B$, we have that $Z \to X_B \to B$ is a
    composition of proper morphisms hence is also proper. Flatness also follows
    from the local description~\eqref{eq:Z-chi}.
\end{proof}

\begin{remark} \label{rmk:can-be-empty-non-flat}
    The scheme $Z_\chi$ defined by an arbitrary $\chi \in \bfB_{X/S,r;K}(T)$ can
    be empty. This can be seen more obviously from the local
    description~\eqref{eq:Z-chi} in Example~\ref{exa:computation}, because of
    the presence of the polynomials $g_{\underline{i}}(\partial_1, \ldots,
    \partial_d)$, $\lvert\underline{i}\rvert \neq 1$.
    For the same reason, unless $d = 1$ or $r = 1$, $Z_\chi/X_T$ is in general
    \emph{not} flat,

    More generally, in~\cite[Thm.~5.1 \& Conj.~5.2]{1905.04741}, it was proved
    that the Hitchin map \eqref{eq:c-dol} factors through a closed subscheme of
    the Hitchin base, and was conjectured that the resulting map is surjective.
    This phenomenon can already be seen from Example~\ref{exa:computation}.
\end{remark}

\subsection{A BNR corespondence for $K$-Higgs bundles}
\label{ssec:BNR-Higgs}

We have already known from the universal property of the universal enveloping
algebra that the category of $K$-Higgs bundles of rank $r$ on $X/S$ is
equivalent to the category of quasi-coherent $\scrO_{\bfV(K)}$-modules whose
direct image on $X$ is locally free of rank $r$. By taking the spectral cover
into consideration, we can get a finer result (cf.~\cite{MR998478}).

\begin{proposition}[{\cite[Theorem~3.2]{zbMATH06666980}}]
    \label{prop:Higgs-spectral-cover}
    Given any $T$-point $\chi$ of $\bfB_{X/S, r; K}$
    with $\iota: Z_{\chi} \into \bfV(K)$, there is an equivalence of
    categories between
    \begin{enumerate}
        \item[(HA)] the fully faithful subcategory $c\inv_{\Higgs} (\chi)$ of
            $\catHiggs_{X/S, r;K}$, and
        \item[(HB)] the fully faithful subcategory of $\catQCoh(Z_{\chi})$
            consisting of objects $M$ such that
            \begin{itemize}
                \item $(\pi_T\circ\iota)_\ast M$ is locally free rank~$r$ on
                    $X_T$, and
                \item the induced Higgs field has characteristic polynomial
                    $\chi$.
            \end{itemize}
    \end{enumerate}
\end{proposition}

\begin{proof}
    To simplify notations, we only work we $T$-points. The arguments below work
    functorially.

    Take an $S$-point $\chi$ of $\bfB_{X/S, r;K}$, i.e., a polynomial
    $\chi(t) = t^r - a_1 t^{r - 1} + \cdots + (-1)^{r}a_r \in \Gamma(X,
    \Sym^\bullet K [t])$ with $a_i \in \Gamma(X, \Sym^i K)$. Suppose that $(E,
    \theta)$ is a Higgs bundle of rank $r$ on $X$ such that the characteristic
    polynomial $\chi_\theta(t)$ of the Higgs field~$\theta$ equals to $\chi(t)$.
    Since $E$ is a Higgs bundle, it gives a quasi-coherent sheaf $\widetilde{E}$
    on $\bfV(K)$. Then the Cayley-Hamilton theorem
    (Proposition~\ref{eq:C-H}) implies that $\widetilde{E}$ is supported on the
    spectral cover $Z_{\chi}$ defined by $\chi$.
    To see this, we only need to verify that $I_\theta\cdot\widetilde{E} = 0$,
    where $I_\theta$ is the sheaf of ideals~\eqref{eq:ideal-of-definition} that
    defines $Z_{\chi}$. In other words, it suffices to show that the morphism
    \[
        \id_{\widetilde{E}} \otimes {(\chi(\lambda))}^\vee: \widetilde{E}
        \otimes_{\scrO_{\bfV(K)}} \pi^\ast (\Sym^r K)^\vee \longrightarrow
        \widetilde{E}
    \]
    is zero, where $\lambda$ is the tautological section of $\pi^\ast K$ and
    $\chi(\lambda) = \det (\lambda - \pi^\ast \theta): \scrO_{V(K)} \to \pi^\ast
    \Sym^r K$. By~\S\ref{ssec:affine-morphism} and projection formula, it
    suffices to show that
    \begin{equation}
        \pi_\ast(\id_{\widetilde{E}} \otimes {(\chi(\lambda)}^\vee): E \otimes
        (\Sym^r K)^\vee \longrightarrow E
        \label{eq:C-H-Higgs}
    \end{equation}
    is zero. One checks easily,
    for example by local computation, that \eqref{eq:C-H-Higgs} is exactly
    \eqref{eq:C-H} composed with
    the evaluation map $\Sym^r K \otimes (\Sym^r K)^\vee
    \to \scrO_X$, hence is zero. Therefore, the Higgs bundle $(E, \theta)$
    with characteristic polynomial $\chi$
    gives rise to a quasi-coherent sheaf on $Z_\chi\subseteq \bfV(K)$.

    Conversely, suppose that $M$ is a quasi-coherent module on $Z_{\chi}$
    defined by $\chi$, with direct image
    $E':=\pi_\ast\iota_\ast M$ to $X$ a locally free
    $\scrO_X$-module of rank $r$. Then $(E', \theta')$ is a Higgs bundle on $X$,
    with the Higgs field $\theta'$ given by the $\scrO_{\bfV(K)}$-module
    structure of $\iota_\ast M$.
\end{proof}

\begin{remark}
    However, it is not clear to us how the characteristic polynomial of
    $\theta'$ is related to $\chi$. Besides, $\tilde E$ is supported on
    $Z_\chi$, but its scheme-theoretic support can be ``thinner'' than $Z_\chi$.
    For example, if $\theta \equiv 0$, then the characteristic polynomial is
    $t^r$. In this case, $Z_\chi$ is ``defined by $t^n = 0$'', but $\tilde E$
    has scheme-theoretic support ``defined by $t = 0$''. In other words, the
    ``minimal polynomial'' may have lower degree than the characteristic
    polynomial.
\end{remark}

\subsection{A BNR correspondence for $L$-bundles}
\label{ssec:BNR-LocSys}

Work in the same settings as in Proposition~\ref{prop:reduced-char-polyn}. In
other words, let $X/S$ be smooth of relative dimension $d>0$. Let $L$ be a
restricted Lie algebroid and set $K:= L^\vee:= \cHom(L, \scrO_X)$. Then thanks
to Proposition~\ref{prop:reduced-char-polyn}, we have a similar result to
Proposition~\ref{prop:Higgs-spectral-cover}.

\begin{proposition}[Groechenig] \label{prop:local-system-spectral-cover}
    Given any $\chi \in \bfB_{X'/S, r;w^\ast K}(T)$, with $\iota: Z_{\chi} \into
    \bfV(w^\ast K)$ the corresponding closed embedding defined by $\chi$, we
    have an equivalence of categories between
    \begin{enumerate}
        \item[(LA)] the fully faithful subcategory $c_{\dR}\inv (\chi)$ of
            $\catVect_{X/S, r; L}$, and
        \item[(LB)] the fully faithful subcategory of
            $\catQCoh(Z_{\chi}, \iota^\ast \scrU(L))$ consisting of objects
            $M$ such that
            \begin{itemize}
                \item the induced $\scrO_{X_T}$-module
                    $\widetilde{\pi_\ast\iota_\ast M}$ is locally free of
                    rank~$r$, and
                \item the induced Higgs filed on $\pi_\ast \iota_\ast M$ has
                    characteristic polynomial $\chi^{p^d}$.
            \end{itemize}
    \end{enumerate}
\end{proposition}

\begin{proof}
    Let $E$ be an $L$-module with $c_{\dR}(E) = \chi$, i.e., an object in (LA).
    Since $E$ is an $U(L)$-module, recall \S\ref{ssec:p-curv} that $F_{X/S,
    \ast}E$ is then an $F_{X/S,\ast}U(L)$-module and it restricts to a
    Higgs-bundle via \eqref{eq:p-curv}, and moreover, by
    Propositoin~\ref{prop:reduced-char-polyn}, the characteristic polynomial of
    this Higgs field is $\chi^{p^d}$. Consider the quasi-coherent
    $\scrO_{\bfV(w^\ast K)}$-module $\widetilde{F_{X/S, \ast}E}$ as in the proof
    of Proposition~\ref{prop:Higgs-spectral-cover}, which is then a $\scrU(L):=
    \widetilde{F_{X/S, \ast}U(L)}$-module too. For this module, we have
    $I_{\chi^{p^d}} \cdot \widetilde{F_{X/S, \ast}E}$ by Cayley-Hamilton. Note
    that $I_{\chi^{p^d}} \subseteq  I_{\chi}^{p^d} \subseteq I_{\chi}$. Hence
    $I_\chi\cdot\widetilde{F_{X/S, \ast}E} = 0$. That is to say,
    $\widetilde{F_{X/S, \ast}E}$ is an object in $\catQCoh(Z_\chi, \iota^\ast
    \scrU(L))$, satisfying properties listed in (LB).

    Conversely, suppose we have a $\iota^\ast \scrU(L)$-module $M$
    on~$Z_{\chi}$ as in (LB). Then $\pi_\ast \iota_{\ast} M$ is an
    $F_{X/S, \ast} U(L)$-module (via the natural map $F_{X/S,\ast} U(L) =
    \pi_\ast \scrU(L) \to \pi_\ast \iota_\ast \iota^\ast \scrU(L)$).
    Therefore, by considering the affine map $F_{X/S}$, we know
    $\widetilde{\pi_\ast \iota_\ast M}$ is a quasi-coherent $\scrO_X$-module as
    well as an $U(L)$-module, because $\widetilde{(F_{X/S})_\ast U(L)} = U(L)$.
    So we get an object in (LA).

    Of course the above arguments work functorially for any $T/S$. It is not
    hard to check the correspondence given above are mutually quasi-inverse to
    each other. So we complete the proof.
\end{proof}

\begin{remark}
    In \cite[Prop.~3.15]{zbMATH06666980} and~\cite[Thm.~2.4]{1707.00752}, they
    use the Morita equivalence and the
    equivalence~\ref{prop:Higgs-spectral-cover} to show the existence of
    $\chi''$ satisfying~\eqref{eq:reduced-char-polyn} and then established the
    equivalence Proposition~\ref{prop:local-system-spectral-cover}.
\end{remark}

\subsection{Azumaya property of the sheaf of crystalline differential operators}

For a general (restricted) Lie algebroid, we could not say much about their
general properties (see \cite{zbMATH06358801} for some of them). However, we
know that in case $L = \leftidx{^1}\cDer_{X/S}$, its universal enveloping
algebra $D_{X/S}$ is \'etale locally modelled by the Weyl algebra. So according
to Theorem~\ref{thm:azumaya-local}, we obtain the following key results, which
can be found in \cite[Lemma~1.3.2, Lemma~2.2.1, Proposition~2.2.2,
Theorem~2.2.3, Remark~2.1.2 and \S2.2.5]{zbMATH05578707}, with $S = \Spec k$ for
an algebraically closed field $k$.

\begin{theorem}\label{thm:azumaya-property}
    Assume that $X/S$ is smooth of relative dimension $d > 0$ and that $S$ is of
    characteristic $p > 0$. Then we have the following facts.
    \begin{enumerate}[wide]
        \item The morphism \eqref{eq:p-curv}
            \begin{equation}
                \psi:\Sym^\bullet \Theta_{X'/S} \longrightarrow
                Z(F_{X/S,\ast} D_{X/S}) \subseteq F_{X/S, \ast} D_{X/S},
                \label{eq:p-curvature-D}
            \end{equation}
            is an isomorphism of $\scrO_{X'}$-algebras. This isomorphism defines
            (according to~\S\ref{ssec:affine-morphism}) a sheaf $\scrD_{X/S}:=
            \scrU(\leftidx{^1}\cDer_{X/S})$ of $\scrO_{\cotsp{X'/S}}$-algebra
            on the cotangent bundle of $X'/S$.
        \item The morphism \eqref{eq:p-curvature-restriction}
            \begin{equation}
                \psi': F_{X/S}^\ast \Sym^\bullet \Theta_{X'/S}
                \longrightarrow Z_{D_{X/S}}(\scrO_X) \subseteq D_{X/S}.
                \label{eq:p-curvature-D'}
            \end{equation}
            is an isomorphism of $\scrO_X$-algebras,
            where $Z_{D_{X/S}}(\scrO_X)$ is the centraliser of $\scrO_X$ in
            $D_{X/S}$.
        \item Moreover, $\scrD_{X/S}$ is an Azumaya algebra of rank $p^{2d}$ over
            the cotangent bundle.
            Actually, there is an isomorphism of
            $(F_{X/S}^\ast \Sym^\bullet \Theta_{X'/S})$-algebras
            \begin{equation}
                F_{X/S}^\ast \left(F_{X/S,\ast} D_{X/S}\right) \simeq
                \cEnd_{F_{X/S}^\ast\left(\Sym^\bullet \Theta_{X'/S}\right)}
                (D_{X/S}),
            \end{equation}
            which defines an splitting of the pullback of
            $\scrD_{X/S}$ to $\cotsp{X'/S} \times_{X', F_{X/S}} X$. This Azumaya
            algebra is nontrivial if the relative dimension $\dim(X/S)$ is more
            than~$0$.

        \item \label{lem:is-invertible}
            Let $\iota: Z \to \cotsp{X/S}$ be any morphism.
            Suppose that we have a splitting
            \[
                \iota^\ast \scrD_{X/S} \simeq \cEnd_{\scrO_{Z}}(P).
            \]
            Then $P$ is a direct image of a rank one locally free sheaf
            $\widetilde{P}$ on $Z \times_{X', F_{X/S}} X$.

        \item \label{lem:split-over-section}
            Let $\iota: X' \to \cotsp{X'/S}$ be a section of the projection
            $\pi: \cotsp{X'/S} \to X'$. Then we have an canonical splitting
            \[
                \iota^\ast \scrD_{X/S} \simeq \cEnd_{\scrO_{X'}} (F_{X/S, \ast}
                \scrO_X).
            \]
    \end{enumerate}
\end{theorem}

\begin{proof}
    Actually this problem is local. Using the ``\'etale coordinates''
    \eqref{eq:etale-coordinate}, we can reduce the problem to the case $\bfA_S^d
    \to S$ with $S$ affine. Then all of these statements have been proven in
    Theorem~\ref{thm:azumaya-local}.
\end{proof}

\begin{remark}
    For a generalisation of this result to smooth algebraic stacks, see
    \cite[Appendix B]{zbMATH06782031}, and for a generalisation to higher level
    differential operators, see~\cite[Prop.~3.6 and Thm.~3.7]{zbMATH05690817}.
\end{remark}

According to Theorem~\ref{thm:azumaya-property}, we have an Azumaya algebra
$\scrD_{X/S}$ on $\cotsp{X'/S}$, so we have the associated $\bbG_m$-gerbe
$\bfS_{\scrD}:= \bfS_{\scrD_{X/S}}$ of splittings over $\cotsp{X'/S}$ as
recalled in~\S\ref{ssec:azumaya-morita}.

\subsection{A Morita equivalence}
\label{sec:splitting-principle}

By comparing Propositions \ref{prop:Higgs-spectral-cover} and
\ref{prop:local-system-spectral-cover}, we may want to further understand the
relation between $\catQCoh(Z_\chi)$ and $\catQCoh(Z_\chi, \iota^\ast \scrU(L))$.
In other words, to understand if $\scrO_{Z_\chi}$ and $\iota^\ast\scrU(L)$ are
Morita equivalent. In case of $L = \cDer_{X/S}$, by virtual of
Theorem~\ref{thm:azumaya-property}, we may conclude the following useful result.

\begin{theorem}[Splitting Principle, {\cite[Lem.~3.27]{zbMATH06666980}}]
    \label{thm:splitting-principle}
    Given $\chi \in \bfB_{X'/S, r}(T)$, denote by $\iota: Z_\chi \into
    \cotsp{X'_T/T}$ and $\pi: \cotsp{X'_T/T} \to X'_T$ as before. Suppose there
    is a splitting
    \[
        \iota^\ast \scrD_{X/S} \simeq \cEnd_{\scrO_{Z_{\chi}}} (P).
    \]
    Then we have an equivalence of categories between
    \begin{itemize}
        \item[(HA)] the category of Higgs bundles of rank $r$ over $X'_T/T$, and
        \item[(LA)] the category of local systems of rank $r$ over $X_T/T$.
    \end{itemize}
\end{theorem}

\begin{proof}
    It suffices to check this for $T$-points. Fix a $\chi \in \bfB_{X'/S,
    r}(S)$. According to Propositions~\ref{prop:Higgs-spectral-cover}
    and~\ref{prop:local-system-spectral-cover}, it suffices to show that
    the corresponding categories (HB) and (LB) are equivalent.
    By assumption, a splitting $\iota^\ast \scrD_{X/S} \simeq
    \cEnd_{\scrO_{Z_\chi}}(P)$ exits. So according to Morita theory
    \cite[(8.12)]{zbMATH05585185}, we have an equivalence of categories
    \[
        \catQCoh(Z_\chi) \simto \catQCoh(Z_\chi, \iota^\ast
        \scrD_{X/S}), \quad
        M \mapsto P \otimes_{\scrO_{Z_\chi}} M.
    \]
    So we only need to show under this equivalence, the extra properties listed
    above correspond to each other. We introduce some temporary notations as in
    the following Cartesian diagram
    \[\begin{tikzcd}
            W_\chi:=Z_\chi \times_{X'} X \ar[r, "\varphi"] \ar[d, "\tau"]
            & Z_\chi  \ar[d, "\pi\circ \iota"] \\
            X \ar[r, "F_{X/S}"] & X'.
    \end{tikzcd}\]
    and fix an object $M$ in $\catQCoh(Z_\chi)$.

    Recall Theorem~\ref{thm:azumaya-property}.\ref{lem:is-invertible} that the
    $\scrO_{W_\chi}$-module $\widetilde P=:L$ is locally free of rank $1$.
    Then we have canonical isomorphisms of $\scrO_X$-modules
    \begin{equation*}
    \begin{aligned}
        \widetilde {(\pi\circ \iota)_\ast(P \otimes_{\scrO_{Z_\chi}} M) }
        & \simeq \tau_\ast (\widetilde{P \otimes_{\scrO_{Z_\chi}} M}) \\
        & \simeq \tau_\ast ( L \otimes_{\scrO_{W_\chi}} \varphi^\ast M)
        & \text{(\cite[Corollaire~9.3.9]{EGAInew})} \\
        & \simeq \tau_\ast L \otimes_{\tau_\ast \scrO_{W_\chi}}
        \tau_\ast \varphi^\ast M \\
        & \simeq \tau_\ast L \otimes_{\tau_\ast \scrO_{W_\chi}}
        F^\ast_{X/S} (\pi\circ \iota)_\ast M.
    \end{aligned}
    \end{equation*}
    Note that $\tau$ is finite because $(\tau \circ \iota)$ is
    (Proposition~\ref{prop:proper-of-universal-spectral-curve}). Then, according
    to \eqref{eq:equivalence-finite-locally-free}, $\tau_\ast L$ is an
    $\tau_\ast\scrO_{W_\chi}$-module locally free of rank~$1$ over $X$.
    Therefore, Zariski locally on $X$, $\widetilde {(\pi\circ \iota)_\ast (P
    \otimes_{\scrO_{Z_\chi}} M)}$ is isomorphic to $F^\ast_{X/S} (\pi\circ
    \iota)_\ast M$. Moreover, $F_{X/S}$ is faithfully flat and locally of finite
    presentation, hence $\widetilde {(\pi\circ \iota)_\ast (P
    \otimes_{\scrO_{Z_\chi}} M)}$ is locally free of rank $r$ if and only if
    $(\pi\circ \iota)_\ast M$ is locally free of rank $r$. Due to the same
    reason, the canonical isomorphism of $\scrO_{X'}$-modules
    \[
        (\pi\circ \iota)_\ast \left(P \otimes_{\scrO_{Z_\chi}} M\right) \simeq
        (\pi\circ \iota)_\ast P \otimes_{(\pi \circ \iota)_\ast\scrO_{Z_\chi}}
        (\pi \circ \iota)_\ast M
    \]
    implies that Zariski locally over $X'$, $(\pi\circ \iota)_\ast \left(P
    \otimes_{\scrO_{Z_\chi}} M\right)$ is isomorphic to a direct sum of $p^d =
    \rk_{\scrO_{Z_\chi}} P$ copies of $(\pi\circ \iota)_\ast M$. Therefore,
    $(\pi \circ \iota)_\ast M$ has characteristic polynomial $\chi$ if and only
    if $(\pi\circ \iota)_\ast \left(P \otimes_{\scrO_{Z_\chi}} M\right)$ has
    characteristic polynomial $\chi^{p^d}$.
\end{proof}

\subsection{Further examples}
In the following examples, we restricts our attention to
$\leftidx{^\lambda}\cDer(X/S)$-modules with $\lambda = 0$ or $1$.

\begin{example}\label{exa:1-form-case}
    In case $r = 1$, given a $\chi \in \bfB_{X'/S,1}(T)$,
    $Z_\chi \into \cotsp{X_T'/T}$
    is isomorphic to the section $X_T' \to \cotsp{X_T'/T}^\ast$ of
    $\cotsp{X_T'/T} \to X_T'$ corresponding to the global section~$\omega$
    determined by~\eqref{eq:section-of-vector-bundles}.
    Recall Theorem~\ref{thm:azumaya-property}.\ref{lem:split-over-section} that,
    the Azumaya algebra $\scrD_{X/S}$ splits over $Z_{\chi}$.
\end{example}

\begin{example} \label{exa:computation-av}
    Suppose that $X/k$ is an abelian variety.
    Then $\bfB$ is representable by a $k$-scheme $B$ as
    in Example~\ref{exa:proper-over-k-representable}.
    Moreover, we know that $\Gamma(X, \scrO_X) = k$ because $X$
    is proper geometrically reduced and geometrically connected,
    and that $\Gamma(X, \Omega_{X/k}^1)$ is a $k$-vector space of dimension
    $d = \dim X$.
    Actually $\Omega_{X/k}^1 \simeq \bigoplus_{i = 1}^d \scrO_{X} \cdot
    \omega_i$, where $\omega_{i} \in \Gamma(X, \Omega_{X/k}^1)$, $i = 1, \ldots,
    d$, are the \emph{invariant differentials}.
    In other words, the sheaf $\Omega_{X/S}^1$ is \emph{globally} trivialised.
    Denote by $\partial_i$, $1 \leq i \leq d$ the $k$-dual of $\omega_i$.
    These $\omega_i$'s (resp.\ $\partial_i$'s)
    form not only a $k$-basis for the $k$-vector space
    $\Gamma(X, \Omega_{X/k}^1)$ (resp.\ ${(\Gamma(X, \Omega_{X/k}^1))}^\vee =
    \Gamma(X, {(\Omega_{X/k}^1)}^\vee)$),
    but also an $\scrO_X$-basis for the free $\scrO_X$-module
    $\Omega_{X/k}^1$ (resp. $(\Omega_{X/k}^1)^\vee$).

    Then apply the computations in Example~\ref{exa:computation}, we know that
    for any $\chi \in B(T)$, we have
    \begin{equation}
        Z_\chi = \cSpec_{\scrO_{X_T}}
        \frac{\scrO_{X_T} [ \partial_1, \ldots, \partial_d]}%
        {(g_1, \ldots, g_d, \ldots, g_{\underline{i}}, \ldots)},
        \label{eq:av-chi}
    \end{equation}
    with $g_{\underline{i}}$ has coefficient in
    $\Gamma(X_T, \scrO_{X_T}) = \Gamma(T, \scrO_T)$ (note $T/k$ is always flat).
    Now for each $\chi \in B(T)$, define
    \begin{equation}
        \widetilde{Z}_\chi = \cSpec_{\scrO_{X_T}}
        \frac{\scrO_{X_T} [ \partial_1, \ldots, \partial_d]}%
        {(g_1, \ldots, g_d)} \supseteq Z_\chi.
        \label{eq:av-chi-tilde}
    \end{equation}
    It is clear that the association of $\chi$ to $\widetilde Z_{\chi}$ is
    functorial, i.e., if $\chi' = \chi \circ \psi$ with $\psi: T' \to T$, then
    $\widetilde Z_{\chi'} = \widetilde X_{\chi} \times_T T'$.
    Note that in this case, $\widetilde{Z}_\chi$
    is always non-empty and is finite and flat over $X$, and $\widetilde
    Z_\chi/B$ is proper, for any $\chi \in \bfB(T)$, cf.
    Remark~\ref{rmk:can-be-empty-non-flat}.

    Naturally, for the universal spectral cover $Z_{\text{univ}}$, we we alse
    have a larger scheme $\tilde{Z}:=\widetilde{Z}_{\text{univ}}\supseteq
    Z_{\text{univ}}$. It has the universal property that for any $\chi \in
    B(T)$, $\widetilde{Z}_{\chi}$ is the pullback of $\widetilde{Z}$.
\end{example}

\begin{example} \label{exa:computation-univ}
    \begin{enumerate}[wide]
        \item \emph{Suppose that $X/S := X/k$ is a smooth proper scheme
            over a field $k$.}
            Hence Example~\ref{exa:proper-over-k-representable} applies.

            Denote by $(\omega_j)_{j \in J}$, the $k$-basis for
            $\bigoplus_{i=1}^r\Gamma(X, \Sym^i\Omega_{X/k}^1)$,
            and by $(\partial_j)_{j \in J}$ the corresponding $k$-dual basis
            for $\bigoplus_{i=1}^r\Gamma(X, \Sym^i\Omega_{X/k}^1)^\vee$.
            Then $\chi_{\text{univ}}$ is actually the tautological section
            \begin{align*}
                \sum_{J \in j} \partial_j \otimes \omega_j
                \in
                \bigg(
                    \Sym^\bullet \big(\bigoplus_{i = 1}^r
                    \Gamma(X, \Sym^i \Omega_{X/k}^1)\big)^\vee
                \bigg)
                \otimes_k
                \bigg(
                    \bigoplus_{i = 1}^r \Gamma(X, \Sym^i \Omega_{X/k}^1)
                \bigg).
            \end{align*}
        \item \emph{Suppose moreover that $X/k$ is a connected group scheme}
            (hence \emph{geometrically connected}).
            In other words, $X/k$ is an abelian variety.
            So Example~\ref{exa:computation-av} applies.

            Moreover, using notations in Example~\ref{exa:computation}, for each
            $m$, $(\underline{\omega}^{\underline{i}})_{|i| = m}$ form a
            $k$-basis for the $k$-vector space $\Gamma(X, \Sym^m\Omega_{X/k}^1)
            = \Sym^m \Gamma(X, \Omega_{X/k}^1)$, as well as the free
            $\scrO_X$-module $\Sym^m \Omega_{X/k}^1$; and
            $(\underline{\partial}^{\underline{i}})_{|i|= m}$ from a $k$-basis
            for the $k$-vector space $\Gamma(X, \Sym^m (\Omega_{X/k}^1)^\vee) =
            \Sym^m \Gamma(X, (\Omega_{X/k}^1)^\vee)$. Denote by
            $\underline{\partial}^{[\underline{i}]}$ the $k$-dual basis for
            ${(\Gamma(X, \Sym^m \Omega_{X/k}^1))}^\vee = \Gamma(X, (\Sym^m
            \Omega_{X/k}^1)^\vee)$. Note that
            $\underline{\partial}^{\underline{i}}$ and
            $\underline{\partial}^{[\underline{i}]}$ are not the same unless
            $|i|=1$.

        \item \emph{Suppose further that the characteristic of $k$ is larger
            than the fixed number $r$ (hence $n!\neq 0$).}

            In this case, for each $1 \leq m \leq r$,
            we have the following ``bad'' isomorphism%
            \footnote{
                The problem is that for any $k$-vector space $V$ and and any
                natural number $n$, the natural paring $ \Sym^n V \times \Sym^n
                (V^\vee) \longrightarrow k$, $(v_1, \ldots, v_n, \ell_1, \ldots,
                \ell_n) \longmapsto \sum_{\sigma \in \frakS_n} \prod_{i = 1}^n
                \ell_i(v_{\sigma(i)})$, is perfect if the characteristic of $k$
                is strictly larger than $n$. Moreover, the \emph{natural} map
                $\Sym^n (V^\vee) \longrightarrow (\Sym^n V)^\vee$ is an
                isomorphism if and only if $n!\neq 0$ in $k$.
                To avoid this problem, the more natural way is to use divided
                power algebras instead of symmetric algebras.
            }
            \begin{equation}
                \begin{tikzcd}[row sep = 0pt,
                    /tikz/column 1/.append style={anchor=base east},
                    /tikz/column 2/.append style={anchor=base west},
                    ]
                    \Gamma\big(X, \Sym^m (\Omega_{X/k}^1)^\vee\big)
                    \ar[r, "\sim"]
                    & \big(\Gamma(X, \Sym^m \Omega_{X/k}^1)\big)^\vee \\
                    \underline{\partial}^{\underline{i}} \ar[r, mapsto]
                    & \dfrac{1}{\underline{i}!}
                    \underline{\partial}^{\underline{i}}
                    =\underline{\partial}^{\underline{[i]}}
                \end{tikzcd}
                \label{eq:bad-isomorphism}
            \end{equation}
            of $k$-vector spaces, and
            $\frac{1}{\underline{i}!}\underline{\partial}^{\underline{i}}$,
            $\lvert\underline{i}\rvert = m$, form a basis form the
            $S(m,d)$-dimensional $k$-vector space on the right hand side.
            Observe that
            \[
                \Gamma(X_B, \Omega_{X_B/B}^1) =
                \bigg(
                    \Sym^\bullet \big(\bigoplus_{i = 1}^r
                    \Gamma(X, \Sym^i \Omega_{X/k}^1)\big)^\vee
                \bigg)
                \otimes_k
                \Gamma(X, \Omega_{X/k}^1).
            \]
            Write $\mu := \partial_1 \otimes \omega_1 + \partial_2 \otimes
            \omega_2 + \cdots + \partial_d \otimes \omega_d \in \Gamma(X_B,
            \Omega_{X_B/B}^1)$. Then using the ``bad'' isomorphism
            \eqref{eq:bad-isomorphism}, we can write
            \begin{align*}
                \chi_{\text{univ}} & = 
                \sum_{m = 1}^{r} \sum_{|\underline{i}| = m}
                \underline{\partial}^{[\underline{i}]} \otimes
                \underline{\omega}^{\underline{i}} \\
                & = \phantom{+}
                \partial_1 \otimes \omega_1 + \partial_2 \otimes \omega_2
                + \cdots + \partial_d \otimes \omega_d \\
                &\phantom{+} + \frac12 \partial_1^2 \otimes \omega_1^2 +
                \cdots + \frac12  \partial_{d}^r \otimes \omega_d^2
                + \sum_{1 \leq i,j \leq d, i \neq j} \partial_{i} \partial_j
                \otimes \omega_i \omega_j\\
                & \phantom{+} + \cdots \\
                & =: \mu + \frac{1}{2} \mu^2 + \cdots + \frac{1}{r!}\mu^r
                \in \bigoplus_{i = 1}^r \Gamma(X_B, \Sym^i \Omega_{X_B/B}^1),
            \end{align*}
            The corresponding characteristic polynomial then can be written as
            \begin{align*}
                \chi_{\text{univ}}(t) & =
                \sum_{i = 0}^r  \frac{{(-1)}^i}{i!} \mu^i t^{r-i} \\
                & = t^r - \mu t^{r-1} + \frac12\mu^2 t^{r-1} - \cdots +
                (-1)^r\frac{1}{r!} \mu^{r}
                \in \Gamma\big(X_B, (\Sym^\bullet\Omega_{X_B/B}^1)[t]\big).
            \end{align*}
        \item \emph{Assume moreover that $k$ is algebraically closed.}

            Now we can even factorise it as
            \[
                \chi_{\text{univ}}(t) =
                (t - c_1 \mu)(t - c_2 \mu) \cdots (t - c_r \mu),
            \]
            for some constants $c_i \in \Gamma(X, \scrO_X) = k$.
    \end{enumerate}
\end{example}

\begin{example}[Flat connections on $\scrO_X$ and their $p$-curvatures]
    \label{exa:p-curvature-of-structure-sheaf}
    It is easy to see that every flat $S$-connection on $\scrO_X$ has the form
    $\rmd + \omega$, where $\rmd: \scrO_X \to \Omega_{X/S}^1$ is the universal
    derivation and $\omega \in \Gamma(X, Z\Omega_{X/S}^1)) = \Gamma(X',
    F_{X/S,\ast}(Z\Omega_{X/S}^1))$ is a \emph{closed} one form, given by
    $\omega = \nabla(1)$. The $p$-curvature, $\psi_{\omega}: \scrO_{X} \to
    F_{X/S}^\ast \Omega_{X'/S}^1$ of the flat connection $(\scrO_X, \rmd +
    \omega)$, identified as a section of $F_{X/S}^\ast\Omega_{X'/S}^1$, is given
    by (see \cite[Proposition~(7.1.2)]{zbMATH03437308}, \cite[201,
    Lemma~4]{zbMATH03149486})
    \[
        F^\ast_{X/S}\left( (w^\ast - C_{X/S})(\omega) \right) \in
        \Gamma(X, F_{X/S}^\ast\Omega_{X'/S}^1).
    \]
    Actually, there is an exact sequence of sheaves of abelian groups on
    $X'_{\etale}$ (\cite[Proposition III.4.14]{zbMATH03674235}),
    \[\begin{tikzcd}
            0 \ar[r]
            & \scrO_{X'}^\ast \ar[r, "F^\ast_{X/S}"]
            & F_{X/S, \ast} \scrO_{X}^\ast \ar[r, "\rmd \log"]
            & F_{X/S, \ast} (Z\Omega_{X/S}^1) \ar[r, "w^\ast - C_{X/S}"]
            & \Omega_{X'/S}^1 \ar[r]
            & 0,
    \end{tikzcd}\]
    where $w: X':= X \times_{X, \Fr_S} S \to X$ is the projection and $C_{X/S}$
    is the Cartier operator in Theorem~\cite[Theorem~7.2]{zbMATH03350023}.
\end{example}

\begin{remark}
    \cite[A.7]{zbMATH06526258} gave a generalisation for the above example.
\end{remark}

\begin{example} \label{exa:surjectivity-of-p-curvature}
    Continue with the previous example and assume that $S = \Spec k$ is the
    spectrum of an algebraically closed field of characteristic $p > 0$, and
    that $X/k$ is an abelian variety. In this case, $w: X' \to X$ is an
    isomorphism, and every global $1$-from is closed.
    So the assignment to flat connection on $\scrO_X$ to its $p$-curvature
    reduces to the map%
    \footnote{
        The map $(w\inv)^\ast C_{X/k}$ is the original Cartier operator
        considered by Cartier in \cite[\S2.6]{zbMATH03149486}, see
        \cite[Remark~7.1.4]{zbMATH03437308}.
    }
    \[
        \id - (w\inv)^\ast C_{X/k}:
        \Gamma\left(X, \Omega_{X/k}^1\right)
        \to \Gamma\left(X, \Omega_{X/k}^1\right).
    \]
    It is then a classical result on $p\inv$-linear maps, see e.g.,
    \cite[Expos\'e III, \frno3, Lemma~3.3]{zbMATH01234891} and \cite[Expos\'e
    XXII, \frno1, Proposition~1.2]{SGA7II}, that this map is surjective. Note
    there that the map $(w\inv)^\ast C_{X/k}$ is $p\inv$-linear, where in the
    mentioned references, $p$-linear maps are discussed; however, the proof
    in~\cite{zbMATH01234891} runs verbatim for $p\inv$-linear maps, as we
    assumed that $k$ is algebraically closed.
\end{example}

\section{A simpson correspondence for abelian varieties}

From now on, assume that $X/S:= X/\Spec k:= X/k$ is an abelian variety over an
\emph{algebraically closed} field $k$ of characteristic $p>0$. Denote by
$e: \Spec k \to X$ the zero section. Then $X'/k$ is again an abelian variety.
Recall Example~\ref{exa:computation-av} that, the Hitchin base
$\bfB':=\bfB_{r}':= \bfB_{X'/k, r}$ is representable by a $k$-scheme
$B':= B'_r:= B_{X'/k, r}$. Moreover, recall
Example~\ref{exa:computation-av} that, we have a universal spectral
cover $Z/X'_{B'}$, and a larger scheme $\widetilde{Z} \supseteq Z$. In case of
$r = 1$, $\widetilde{Z} = Z$.

\subsection{Rank one case}
This result is due to Roman Bezrukavnikov, see \cite[Thm.~4.14]{zbMATH05277424}
and \cite[Appendices.~B and C]{zbMATH06782031}. We reproduce it as follows.

Noting that when $r = 1$, we know that the Hitchin base
\[
    B':=B'_1= \bfV(\Gamma(X', \Omega_{X'/k}^1))
\]
is the $k$-vector space of global sections of $\Omega_{X'/k}^1$. Hence a
$T$-point of $B'$ is the same as a one-form $\omega$ on $X_T'$. Any such
$\omega$ determines a closed subscheme $Z_{\omega}$ of the cotangent bundle of
$X_T'/T$. In fact, recall Example~\ref{exa:1-form-case} that, the inclusion of
$Z_\omega$ into the cotangent bundle is the same as a section of the projection
from the cotangent bundle to $X'_T$. In particular, we may identify the
inclusion of the universal spectral cover $Z$ into $\cotsp{X'_{B'}/B'}$ as the
section $X'_{B'} \to \cotsp{X'_{B'}/B'}$ determined by the $1$-form
$\chi_{\text{univ}} \in \Gamma(X'_{B'}, \Omega_{X'\times {B'}/B'}^1)$.
Meanwhile, in this case, $\catLoc_{X/k, 1}$ is usually denoted by
$\catPic^\natural_{X/k}$.
That is, for any $k$-scheme $T$, $\catPic^\natural_{X/k}(T)$ is the groupo\"id
of invertible sheaves with flat $T$-connections on $X_T$,
i.e, of $D_{X_T/T}$-modules that are invertible as $\scrO_{X_T}$-modules.
Let $\catPic_{X/k}$ be the (relative) \emph{Picard stack}%
\footnote{
    Picard stack is sometimes a confusing name: it can refer to a
    \emph{(strictly) commutative group stack},
    see \cite[Expos\'e XVIII, \S1.4]{SGA4III}.
}
of invertible sheaves, i.e., $\catPic_{X/k}(T)$ is the groupo\"id of invertible
sheaves on $X_T$ for any $T/k$. In other words, $\catPic_{X/k} = f_\ast(\rmB
\bbG_{\rmm, X}) = \Res_{X/k}(\rmB \bbG_{\rmm, X})$
(see \cite[Expos\'e XVIII, \S\S1.4.21 and 1.5.1]{SGA4III}).

For any scheme $T/k$, a \emph{rigidified invertible sheaf} $(L, \alpha)$
consists of an invertible sheaf $L$ on $X_T$ together with an isomorphism
$\alpha: e_T^\ast L \simeq \scrO_T$, where $e_T$ is the pullback of
the unit section $e$. A flat connection on $(L, \alpha)$ is just a flat
$T$-connection on $L$. Then there are \emph{$k$-group schemes} $\Pic_{X/k, e}$
and $\Pic_{X/k, e}^\natural$ over $k$ satisfying that for any $T/k$,
\begin{align*}
    \Pic_{X/k, e}(T) & = \{
        \text{isom.\ classes of rigidified invertible sheaves }
        (L, \alpha) \text{ over } X_T
    \}, \\
    \Pic_{X/k, e}^\natural(T) & = \{
        \text{isom.\ classes of rigidified invertible sheaves} \\
        & \pushright{\text{with a flat connection } (L, \alpha, \nabla)
        \text{ over } X_T}
    \}.
\end{align*}
The existence of the scheme $\Pic_{X/k, e}$ (i.e., the representability of the
associated fppf-sheaf of the functor as above) is the classical theory on Picard
functors, see \cite{MR2223410} for a detailed exposition;
while the existence of the scheme $\Pic_{X/k, e}^\natural$ is discussed
in~\cite{zbMATH03470573} and see also~\cite[Proposition~4.11]{zbMATH05277424}
and \cite[Appendix~B]{zbMATH05765015}.
Moreover, $\catPic_{X/k} = \Pic_{X/k} \times_k \bfB\bbG_{\rmm}$,
and $\catPic_{X/k}^\natural = \Pic_{X/k}^\natural \times_k \bfB\bbG_{\rmm}$.
Clearly there are natural maps $\catPic_{X/k}^\natural \to \catPic_{X/k}$
and $\Pic_{X/S,e}^\natural \to \Pic_{X/k, e}$, compatible with the projections.

\begin{proposition}
    The map
    \[
        c_{\dR}: \catPic_{X/k}^\natural \longrightarrow B'_1
    \]
    assigning to each invertible sheaf with a flat connection its $p$-curvature
    as defined in Proposition~\ref{prop:reduced-char-polyn} induces a $k$-group
    scheme homomorphism,
    \[
        c_{\dR}: \Pic_{X/k, e}^\natural \longrightarrow B'_1.
    \]
    Moreover, $\scrD_{X/k}$ splits canonically over $Z_{c_{\dR}}
    \subseteq \cotsp{X' \times \Pic^\natural/\Pic^{\natural}}$.
    In other words, once pulled back along the composition of natural maps
    (with the identification in Example~\ref{exa:1-form-case})
    \[
        \begin{tikzcd}[column sep = small]
            Z_{c_{\dR}} \simeq X' \times_k \Pic^\natural_{X/k, e}
            \ar[r, "\id \times c_{\dR}"] &[1em]
            X' \times_k B' \simeq Z \ar[r, hook] &
            \cotsp{X' \times B'/B'} \ar[r] &
            \cotsp{X'/k},
        \end{tikzcd}
    \]
    the Azumaya algebra $\scrD_{X/k}$ splits.
\end{proposition}

\begin{proof}
    This scheme version is given~\cite[\S4.3]{zbMATH05277424} and
    a stack version is given in~\cite[Proposition~B.3.4]{zbMATH06782031}.

    Actually, this follows directly from the fact that if $A$ s an Azumaya
    algebra of rank $r^2$ as an $\scrO_X$-module, and if $M$ is an $A$-module
    locally free of rank $r$ as an $\scrO_X$-module, then $M$ is a splitting
    module of $A$.
\end{proof}

\begin{corollary} \label{cor:pic-equiv-to-splitting}
    There is an equivalence $\catPic^\natural_{X/k} \simeq
    \Res_{Z/B'}(\bfS_{\scrD}|_{Z})$ of stacks over $B'$. And in particular,
    $\Res_{Z/B'}(\bfS_{\scrD}|_Z)$ is algebraic.
\end{corollary}

\begin{proof}
    This is just a stack version \cite[Proposition~4.13, 1)]{zbMATH05277424},
    which proved a \emph{rigiedfied} version of this proposition. Recall that,
    for any $\omega \in B'(T)$, $\catPic^\natural_{X/k}(T)$ is the category of
    invertible sheaves on $X_T$ with a flat $T$-connection, and that
    $\Res_{Z/B'}(\bfS_{\scrD}|_Z)(T)$ is the category of splittings of the pull
    back of $\scrD_{X/k}$ to $Z_\omega \subseteq \cotsp{X'\times T/T}$. Write
    $\iota: Z_\omega \into \cotsp{X'\times T/ T}$ and $\pi_T: \cotsp{X'\times
    T/T} \to X'_T$ for the inclusion and the projection.

    In this case, we have (Example~\ref{exa:1-form-case}) that $\pi_T \circ
    \iota$ is an isomorphism. Let $M$ be $\scrO_{Z_\omega}$, which defines a
    rank $1$ Higgs bundle on $X'$. Then according to the arguments in the
    splitting principle in \S\ref{sec:splitting-principle}, we know that
    \begin{align*}
        \Res_{Z/B'}(\bfS_{\scrD}|_Z)(T)
        & \longrightarrow \catPic_{X/k}^\natural(T)
        = \catLoc_{X/k, 1}(T) \\
        P & \longmapsto \widetilde{(\pi_T\circ \iota)_\ast (P)}
    \end{align*}
    is an equivalence.

    Algebraicity of $\Res_{Z/B'}(\bfS_{\scrD}|Z)$ follows from that
    $\catPic_{X/k}^\natural$ is algebraic. This also follows from
    \cite[Thm.~1.5]{zbMATH05049029}, because of the fact that $Z/B'$ is proper
    and flat.
\end{proof}

\begin{remark}
    In \cite{zbMATH06782031}, It is further proved that $X^\natural \simeq
    \scrT_{\scrD}$, where
    \[
        X^\natural = \Pic_{X/k}^\natural \times_{\Pic_{X/k}} X^\vee
    \]
    is the \emph{universal (vector) extension} of $X$
    \cite{zbMATH03470573}, and $\scrT_{\scrD}$ is the stack of \emph{tensor
    splittings}, or \emph{multiplicative splittings} in the terminology of
    \cite{zbMATH05277424}.
\end{remark}

\begin{proposition} \label{prop:smooth-and-surjective}
For abelian varieties, the morphism tangent map
\[
    c_{\dR}: \catPic_{X/k}^\natural \to B'_1
\]
is \emph{smooth} and \emph{surjective}. In particular, it is formally
smooth.
\end{proposition}

\begin{proof}
    Smoothness follows from the fact that
    \[
        \rmd c_{\dR}:
        \bfV(\rmH^1_{\dR}(X/k)) \to \bfV(\Gamma(X', \Omega_{X'/k}^1))
    \]
    is surjective (see \cite[Thm.~4.14]{zbMATH05277424}). Surjectivity is a
    consequence of Example~\ref{exa:surjectivity-of-p-curvature}.
\end{proof}

\begin{corollary}[Bezrukavnikov]
    \label{thm:splits-1-form}
    The Azumaya algebra $\scrD_{X/S}$ splits over the \emph{formal
    neighbourhood} of each 1-form (cf.~\cite[Thm~4.14]{zbMATH05277424}).
\end{corollary}

\begin{proof}
    Recall Example~\ref{exa:1-form-case} that $\scrD_{X/S}$ splits over the
    graph of $1$-forms. Using the formal smoothness of $c_{\dR}$ in
    Proposition~\ref{prop:smooth-and-surjective}, as well as the equivalence in
    Corollary~\ref{cor:pic-equiv-to-splitting}, we can conclude that the
    splitting lifts to the formal neighbourhood.
\end{proof}

\subsection{Higher rank case}
\label{ssec:Higer-rank-case}
Now we deal with the higher rank case. To this aim, we will mainly use the
computations that we have done in Examples~\ref{exa:computation},
\ref{exa:computation-av} and~\ref{exa:computation-univ}. In particular,
recall~\eqref{eq:av-chi}, for any $\chi\in B'_r(k)$, we know that $Z_\chi$ is
the closed subscheme of the cotangent bundle of ${X'_T/T}$, cut out by the
$S(r, d)$ equations $g_1, \ldots, g_d, \ldots$ appeared as coefficients of
$(\pi^\ast \underline{\omega}^{\underline{i}})$
in~\eqref{eq:local-defining-equation}. And we defined a larger closed
subscheme~\eqref{eq:av-chi-tilde} that is cut out by $d$ polynomials
$g_1, \ldots, g_d$, each of which is a polynomial in only one variable, that
has coefficients in $\Gamma(X', \scrO_X') = k$. Let $Z/X'_{B'}$ and
$\widetilde Z/X'_{B'}$ be the universal families as defined
in~Example~\ref{exa:computation-av}.

\begin{corollary} \label{cor:splits-in-general}
    For any $\chi \in B'_r(k)$, $\scrD_{X/S}$ splits over the formal
    neighbourhood of $\widetilde Z_\chi$ (hence over that of $Z_\chi$ if $Z_\chi
    \neq \emptyset$).
\end{corollary}

\begin{proof}
    Recall that $\widetilde{Z}_\chi$ is cut out by $d$ polynomials $g_i \in
    k[\partial_i]$, $1 \leq i \leq d$. Since $k$ is \emph{algebraically closed}
    by assumption, each $g_i$ factors as a product $\prod_{m = 1}^r (\partial_i
    - c_{i, m})$, $c_{i, m} \in k$. So we can conclude that $\widetilde{Z}_\chi$
    is a union of (possibly non-reduced) closed subschemes of the $j$-th
    infinitesimal neighbourhood of graphs of some $1$-forms, $j \leq r$. Hence,
    according to Corollary~\ref{thm:splits-1-form}, the Azumaya algebra
    $\scrD_{X/k}$ splits over the formal neighbourhood of $\widetilde{Z}_\chi$,
    \textit{a fortiori}, over that of $Z_\chi$.
\end{proof}

\begin{proposition} \label{prop:stack-of-splittings-is-algebraic}
    Let $\bfS:=\Res_{\widetilde{Z}/B'}(\bfS_{\scrD}|_{\widetilde{Z}}) \to B'$ be
    the stack of splittings $\scrD_{X/k}$ relative to $\widetilde{Z}/B'$. Then
    the stack $\bfS/B'$ is algebraic, and it is smooth and surjective over $B'$.
    Moreover, $\bfS$ is a $\catPic_{\widetilde{Z}/B'}$-torsor.
\end{proposition}

\begin{proof}
    Recall Remark~\ref{rmk:can-be-empty-non-flat} that $\widetilde{Z}$ is
    \emph{proper} and \emph{flat} over $B'$. So according to
    \cite[Thm.~1.5]{zbMATH05049029}, the Weil restriction
    $\bfS:=\Res_{\widetilde{Z}/B'}(\bfS_{\scrD}|_{\widetilde{Z}})$ is
    \emph{algebraic} and \emph{locally of finite presentation}. Hence, according
    to \cite[\SPtag{0DP0}]{stacks-project}, to show that $\bfS/B'$ is smooth it
    suffices to show that $\bfS/B'$ is formally smooth. Note that $B'$ is
    (locally) noetherian, $\bfS/B'$ is locally of finite type, and $k$ is
    algebraically closed, then according to \cite[\SPtag{02HY}]{stacks-project},
    it suffices to show that the Azumaya algebra $\scrD_{X/k}$ splits over the
    formal neighbourhood $Z_\chi$ for all $\chi \in B'(k)$, which is exactly
    Corollary~\ref{cor:splits-in-general}. The surjectivity also follows.
    Therefore, \'etale locally on $B'$, $\bfS$ admits a section, or more
    precisely, there is an \'etale surjective morphism $U \to B'$, such that
    $\bfS(U)$ is non-empty. Since $\bfS_{\scrD}|_{\widetilde{Z}}$ is a
    $\bbG_{\rmm,\widetilde{Z}}$-gerbe, i.e., a $\rmB \bbG_{\rmm,
    \widetilde{Z}}$-torsor, so
    $\bfS:=\Res_{\widetilde{Z}/B'}(\bfS_{\scrD}|_{\widetilde{Z}})$ is a pseudo
    $\Res_{\widetilde{Z}/B'}(\rmB \bbG_{\rmm, \widetilde{Z}})$-torsor, i.e., a
    pseudo $\catPic_{\widetilde{Z}/B'}$-torsor. The existence of an \'etale
    local section implies that it is actually a torsor. Hence $\bfS$ is an
    $\catPic_{\widetilde{Z}/B'}$-torsor.
\end{proof}

\subsection{Main result}
Recall Proposition~\ref{prop:stack-of-splittings-is-algebraic} we have an
$\catPic_{\widetilde{Z}/B'}$-torsor
\[
    \bfS:=\Res_{\widetilde{Z}/B'} (\bfS_{\scrD}|_{\widetilde{Z}}).
\]
Note moreover that, via the identifications in
Proposition~\ref{prop:Higgs-spectral-cover} and
Proposition~\ref{prop:local-system-spectral-cover}, tensor products define
actions
\begin{align*}
    \catPic_{\widetilde{Z}/B'} \times_{B'} \catHiggs_{X'/k,r}
    \longrightarrow \catHiggs_{X'/k, r},\\
    \quad\text{and}\quad
    \catPic_{\widetilde{Z}/B'} \times_{B'} \catLoc_{X/k,r}
    \longrightarrow \catLoc_{X/k, r}
\end{align*}
of $\catPic_{\widetilde{Z}/B'}$ on $\catHiggs_{X'/k,r}$ and $\catLoc_{X/k,r}$
respectively over $B'$. Verifications of such actions are well defined, in
particular on $\catLoc_{X/k, r}$, are similar to the arguments in
\S\ref{sec:splitting-principle} (cf.~\cite[Proposition~3.5]{zbMATH06526258}).
The formulation of the following theorem is inspired by that of
\cite[Theorem~1.2, Remark~3.13]{zbMATH06526258}.

\begin{theorem} \label{thm:main}
    There is a $\catPic_{\widetilde{Z}/B'}$-equivariant isomorphism of stacks
    \[
        C\inv_{X/k}: \bfS \times^{\catPic_{\widetilde{Z}/B'}} \catHiggs_{X'/k,
        r} \longrightarrow \catLoc_{X/k, r}
    \]
    over $B'$. In particular, there is an \'etale surjective morphism $U\to
    \bfB'$, such that
    \[
        \catHiggs_{X'/k, r}\times_{B'} U  \simeq \catLoc_{X/k, r} \times_{B'} U.
    \]
\end{theorem}

\begin{proof}
    The first statement follows from the splitting principle described in
    \S\ref{sec:splitting-principle} and Corollary~\ref{cor:splits-in-general}.
    In fact the map is given as follows. For any $\chi \in B'(T)$, denote by
    \[\begin{tikzcd}
            \iota:Z_{\chi} \ar[r, closed, "\gamma"] &
            \widetilde{Z}_{\chi} \ar[r, closed, "\tilde\iota"] &
            \cotsp{X'_T/T},
    \end{tikzcd}\]
    the inclusions. For any object $(E, \theta)$ in $\catHiggs_{X'/k}(T)$,
    consider via Proposition~\ref{prop:Higgs-spectral-cover} the quasi-coherent
    sheaf $\tilde E$ on $Z_\chi$. Any object
    \[
        (\chi: T \to B', P,
        \alpha: \tilde\iota^\ast\scrD_{X_T/T} \simeq
        \cEnd_{\scrO_{\widetilde{Z}_\chi}}(P))
    \]
    in $\bfS(T)$ defines a splitting module $\gamma^\ast P$ of $\iota^\ast
    \scrD_{X_T/T}$ on $Z_\chi$. Then the $\scrO_{X_T}$-module
    \[
        C_{X/k}\inv(E):= \widetilde{(\pi_T\circ \iota)_\ast(\widetilde{E}
        \otimes_{\scrO_{Z_\chi}} \gamma^\ast P)}.
    \]
    with the notation as in \S\ref{sec:splitting-principle}, is an
    $D_{X_T/T}$-module, i.e., an object in $\catLoc_{X/S}(T)$. Then clearly the
    assignment $((\chi, P, \alpha), (E, \theta)) \mapsto C_{X/k}\inv(E)$ defines
    a $\catPic_{\widetilde{Z}/B'}$-equivariant map $\bfS
    \times^{\catPic_{\widetilde{Z}/B'}} \catHiggs_{X'/k, r} \to \catLoc_{X/k,
    r}$. This is an isomorphism follows directly from the discussion in
    \S\ref{sec:splitting-principle}. The second part follows from Proposition
    \ref{prop:stack-of-splittings-is-algebraic} that there is an \'etale cover
    $U \to B'$, such that $\bfS \times_{B'} U \simeq
    \catPic_{\widetilde{Z}_\chi/B'} \times_{B'} U$.
\end{proof}

\printbibliography[notcategory={sigles}]
\printbibliography[category = {sigles}, title={Abbreviation}, heading =
subbibliography]

\end{document}